\documentclass[12pt, twoside, reqno]{article}

\usepackage{amsmath,amsthm}
\usepackage{amssymb}
\usepackage{bm}

\usepackage{enumitem}

\usepackage{graphicx}

\usepackage[T1]{fontenc}

\newcommand{\D}{\mathrm{d}}
\newcommand{\ee}{\mathrm{e}}
\newcommand{\ii}{\mathrm{i}}
\allowdisplaybreaks
\RequirePackage{fix-cm}
\newtheorem{assumption}{Assumption}
\newtheorem{notation}{Notations}
\newtheorem{lemma}{Lemma}
\newtheorem{defi}{Definition}
\newtheorem{exam}{Example}
\newtheorem{theorem}{Theorem}
\newtheorem{remark}{Remark}

\numberwithin{equation}{section}

\begin{document}


\baselineskip=17pt


\title{Asymptotic expansions for the solution of a linear PDE 
	with a multifrequency  highly oscillatory potential}

\author{Antoni Augustynowicz\\
Institute of Mathematics\\ 
Faculty of Mathematics, Physics and Informatics,\\
University of Gda\'{n}sk\\
Wita Stwosza 57\\
80-308 Gda\'{n}sk, Poland\\
E-mail: antoni.augustynowicz@ug.edu.pl
\and 
Rafa{\l} Perczy\'{n}ski\\
Institute of Mathematics\\ 
Faculty of Mathematics, Physics and Informatics,\\
University of Gda\'{n}sk\\
Wita Stwosza 57\\
80-308 Gda\'{n}sk, Poland\\
E-mail: rafal.perczynski@phdstud.ug.edu.pl}

\date{}

\maketitle


\renewcommand{\thefootnote}{}



\renewcommand{\thefootnote}{\arabic{footnote}}
\setcounter{footnote}{0}


\begin{abstract}
Highly oscillatory differential equations present significant challenges in numerical treatments. The
Modulated Fourier Expansion (MFE), used as an ansatz, is a commonly employed tool as a numerical
approximation method. 
In this article, the Modulated Fourier Expansion is analytically derived for a linear partial differential equation with a multifrequency highly oscillatory potential.
The solution of the equation is expressed as a convergent Neumann series within the appropriate Sobolev space.
The proposed approach enables, firstly, to derive a general formula for the error associated with the
approximation of the solution by MFE, and secondly, to determine the coefficients for this expansion
-- without the need to numerically solve  the system of differential equations to find the coefficients of
MFE. Numerical experiments illustrate the theoretical investigations.
\end{abstract}

\section{Introduction}
\label{sec:sample1}
We consider the following highly oscillatory partial differential equation
\begin{align}
\label{eq:1.5}
&\partial_tu(x,t)=\mathcal{L}u(x,t)+f(x,t)u(x,t),\qquad t\in [0,t^\star],\quad x\in\Omega\subset\mathbb{R}^m,\\ \nonumber
&u(x,0)=u_0(x),
\end{align}
with zero boundary conditions, where $\Omega$ is an open and bounded subset of $\mathbb{R}^m$  with smooth boundary $\partial \Omega$, $t^\star>0$ and $\mathcal{L}$ is a linear differential operator of degree $2p$, $p\in \mathbb{N}$, defined by the formula
\begin{align} \label{diff_operator}
\mathcal{L} = \sum_{|\bm{p}|\leq 2p}^{}a_{\bm{p}}(x) D^{\bm{p}},  \qquad D^{\bm{p}}= \frac{\partial^{p_1}}{\partial x_1^{p_1}}\frac{\partial^{p_2}}{\partial x_2^{p_2}}\dots \frac{\partial^{p_m}}{\partial x_m^{p_m}}, \quad  x\in \Omega\subset \mathbb{R}^m.
\end{align}
Multi-index $\bm{p}$ is an $m$-tuple of nonnegative integers $\bm{p} = (p_1,p_2,\dots,p_m)$ and $a_{\bm{p}}(x)$ are smooth, complex-valued functions of $x \in \bar{\Omega}$.
We assume that function $f(x,t)$ in  equation (\ref{eq:1.5}) is a highly oscillatory of type
\begin{equation}\label{ideal_function_f}
f(x,t) = \sum_{n=1}^{N}\alpha_n(x)e^{i n\omega  t}, \quad \omega \gg 1, \quad N\in \mathbb{N},
\end{equation}
where $\alpha_n$ are complex-valued, sufficiently smooth functions.
The proposed method can also be efficiently applied  in the case of  functions $\alpha_n$ depending on the time variable $t$. This will be demonstrated in numerical experiments.
This paper aims to express the solution of  (\ref{eq:1.5}) as the following partial sum of the asymptotic expansion 
\begin{equation}\label{partial_expansion}
u(x,t) = p_{0,0}(x,t) +  \sum_{r=1}^{R}\frac{1}{\omega^r}\sum_{s=0}^{S}p_{r,s}(x,t)e^{i s \omega t}+\frac{1}{\omega^{R+1}}E_{R,S}(x,t), 
\quad t\in [0,t^\star],\quad x\in\Omega\subset\mathbb{R}^m,
\end{equation}
where coefficients $p_{r,s}$ and the magnitude of error $E_{R,S}$ are independent of parameter $\omega$.
Needless to say, the general form of equation (\ref{eq:1.5}) encompasses many important equations from both classical and quantum physics. The most important examples include the heat equation and the Schr\"{o}dinger equation. We demonstrate that the proposed method can also be applied  to equations with second-order time derivatives, such as the wave equation and the Klein-Gordon equation.  Highly oscillatory differential equations of type (\ref{eq:1.5}) arise in various fields, including electronic engineering \cite{Condon_2009,condon_2012}, when computing scattering frequencies \cite{Condon_2019}, and in quantum mechanics \cite{IKS_2018,CKLP_2021}.


The asymptotic expansion of type (\ref{partial_expansion}), also known as a Modulated Fourier expansion or frequency expansion, is an important technique in computational mathematics used, \emph{inter alia}, for analysing  highly oscillatory  Hamiltonian systems over long times \cite{Lubich_2001,Lubich_2003} and in the study of the heterogeneous multiscale method for oscillatory ordinary differential equations \cite{Sanz-Serna_2009}.   A comprehensive and detailed description of the Modulated Fourier expansion can be found in \cite{Hairer}. Furthermore, the Modulated Fourier expansion can be utilized in numerical-asymptotic approaches as an ansatz for solving linear or nonlinear highly oscillatory differential equations 
\cite{cohen_2004,CKLP_2021,Condon_2010,condon_2012}
In short, this ansatz is incorporated into the equation, and subsequently, the coefficients $p_{r,s}$ in the sum (\ref{partial_expansion}) are determined either recursively or numerically by solving non-oscillatory differential equations. This approach allows us to approximate highly oscillatory equations with great accuracy.
It is known, that the sum (\ref{partial_expansion}) approximates the solution with error $C\omega^{-R-1}$, but the constant $C$ is unknown. We derive the formula for this constant. This may be important in practise since this constant can be large.

In this paper, instead of employing an ansatz, we approximate the solution to equation (\ref{eq:1.5}) by deriving a partial sum of the asymptotic expansion (\ref{partial_expansion}) purely analytically. This approach enables us to obtain formulas for the coefficients $p_{r,n}$ of  sum (\ref{partial_expansion}), eliminating the need to determine them by solving a system of differential equations. Furthermore, this approach allows us to derive a formula for the error in the approximation by the asymptotic sum (\ref{partial_expansion}).

To express the solution of (\ref{eq:1.5}) as an asymptotic series, we intend to use the computational methods proposed for computing highly oscillatory integrals. To find the approximate solution of the partial differential equation, firstly, we show that the Neumann series -- in other words, a series of multivariate integrals -- converges to the solution of equation (\ref{eq:1.5}) in the Sobolev space $H^{2p}(\Omega)$, where $2p$ is the order of the differential operator $\mathcal{L}$ defined in (\ref{diff_operator}).
Then, by using integration by parts and the theory of semigroups, we expand asymptotically each of these integrals  into a sum of known coefficients.
This is the most complicated and technical part, since the domain of each integral is a $d$-dimensional simplex, $d=1,2,\dots$  and, as a result, the number of terms in the asymptotic expansion grows exponentially as $d$ increases. 
Fortunately, our numerical experiments show that one can achieve a small enough error even for relatively small values of $d$. By organizing terms appropriately with respect to magnitudes $\omega^{-r}$
and frequencies $\ee^{\ii s\omega t}$, 
we obtain the sum (\ref{partial_expansion}). Given that the Neumann series converges for any time variable $t$, we attain a long-time behavior of the highly oscillatory solution to the equation.


By considering the potential function $f$ in the form (\ref{ideal_function_f}), we  avoid the occurrence of resonance points in the highly oscillatory integrals that appear in the Neumann series. In the paper \cite{KKRP}, the authors present the asymptotic expansion of the solution for a highly oscillatory equation with a single frequency, specifically, when $f$ is given by $f(x,t)=\alpha(x) \ee^{\ii\omega t}$.
This paper is a continuation of the research commenced in \cite{KKRP}  and is the next step towards approximating  the solution of equation (\ref{eq:1.5})  with a more general potential of the following form
\begin{eqnarray}\label{ideal_function_f2}
f(x,t) = \sum_{\substack{n=-N \\ n\neq 0}}^{N}\alpha_n(x,t)e^{i n\omega  t}, \quad \omega \gg 1, \quad N\in \mathbb{N}.
\end{eqnarray}
For such  a function, resonance points will appear repeatedly  in the integrals which constitute the Neumann series. In this context,
resonance occurs when the vector $\bm{n} \in \mathbb{Z}^d$ in the frequency term $\ee^{\ii\bm{n}^T \bm{\tau}}$, where $\bm{\tau}\in \mathbb{R}^d$, is orthogonal to the boundary of the integral $\int_{\sigma_d(t)} f(\bm{\tau})\ee^{\ii\bm{n}^T \bm{\tau}}\D \bm{\tau}$, with $\sigma_d(t)$ representing a $d$-dimensional simplex.
The presence of resonance points renders the approximation of high oscillation integrals even more difficult.
In this manuscript, we consider a special case of such a situation and demonstrate that if the vector $\bm{n}$ is orthogonal to one edge of the simplex $\sigma_d(t)$, then resonance points vanish in the sum of  integrals from the Neumann series.

The results presented in this paper extend those previously introduced in \cite{KKRP}. Specifically, we consider a potential function $f$ with multifrequencies (\ref{ideal_function_f}), rather than one with a single frequency. Furthermore, we establish the convergence of the Neumann series in a more restrictive Sobolev norm, instead of the $L^2$ norm, namely we provide the convergence in the norm of Sobolev space $H^{2p}(\Omega)$. In addition, we show that the proposed approach for computing highly oscillatory equation can be applied to equations featuring second-order time derivatives. The resonance case is also considered.

The paper is organized as follows: in Section \ref{sec2} we show that the solution of (\ref{eq:1.5}) can be expressed as the Neumann series.  In Section \ref{sec3}, we convert each term of the Neumann series into a sum of multivariate highly oscillatory integrals. Then we introduce necessary definitions which will be needed in Section \ref{sec4} for the asymptotic expansion of each highly oscillatory integral. Section \ref{sec5} is devoted to the error analysis of the proposed approximation.  Section \ref{sec6} concerns highly oscillatory integrals with resonance points which form the Neumann series. In Section \ref{sec7},  we show the utility of the asymptotic method to the wave equation. Numerical simulations are presented in Section \ref{sec8},   while the conclusion and ideas for the future research are discussed in Section \ref{sec9}.

\section{Representation of the solution as the Neumann series}\label{sec2}
In this section, we show that the solution of (\ref{eq:1.5}) can be presented
as the Neumann series for any $t>0$.
We start by introducing the necessary notations and making a general assumption, which will be needed throughout the text. 
\begin{notation}
	By  $ H^{2p}(\Omega) =W^{2p,2}(\Omega) $, where $p$ is a nonnegative integer, we understand the Sobolev space equipped with standard norm $\| \ \ \|_{H^{2p}(\Omega)}$, and $H^{p}_0(\Omega)$ is the closure of $C_0^{\infty}(\Omega)$ in the space $H^p(\Omega)$. The norm in the space $D(\mathcal{L}):=H_0^p(\Omega)\cap H^{2p}(\Omega)$ is denoted by $\| \ \ \|:=\| \ \ \|_{H^{2p}(\Omega)} $.
	Additionally, we slightly abuse the notation and also denote $u(t):=u( \cdot  ,t)$ as an element of an appropriate Banach space.
\end{notation}
\noindent The following assumption will apply throughout the text.
\begin{assumption} \label{assumption1}
	Suppose that
	\begin{enumerate}
		\item $\Omega $ is an open and bounded set in $\mathbb{R}^m$ with smooth boundary $\partial \Omega$.
		\item Operator  $-\mathcal{L}: D(\mathcal{L}):= H_0^p(\Omega)\cap H^{2p}(\Omega) \rightarrow L^2(\Omega)$, where $\mathcal{L}$  is of form (\ref{diff_operator}), is a strongly elliptic of order $2p$   and has smooth, complex-valued coefficients $a_{\textbf{p}}(x)$ on $\bar{\Omega}$.
		\item $u_0 \in D(\mathcal{L})$ and $f\in C\left([0,t^\star];H^{2p}(\Omega)\right)$.
	\end{enumerate}
\end{assumption}
We emphasize that the space $D(\mathcal{L})$ is the Banach space since it is a closed subspace of the Banach space $H^{2p}(\Omega)$.

Expressing the solution $u(t)$ of equation (\ref{eq:1.5}) with the highly oscillatory potential (\ref{ideal_function_f}) as the Neumann series facilitates its numerical approximation.
Namely, having expressed  $u(t)$ as the following series 
$$u(t) = \sum_{d=0}^{\infty}T^d e^{t\mathcal{L}}u_0,$$
for certain linear operator $T$, each term $T^d e^{t\mathcal{L}}u_0$ is actually a sum of multivariate highly oscillatory integrals. In the following, we will show that if  such an  integral  satisfies the nonresonance condition, then $T^d e^{t\mathcal{L}}u_0\sim \mathcal{O}(\omega^{-d})$.    Intuitively, for a large parameter $\omega \gg 1$, further terms of the  Neumann series become less relevant in the numerical approximation.

We start by applying Duhamel's formula, and we write equation (\ref{eq:1.5}) in the integral form
\begin{eqnarray}\label{eq:Duhamel}
u(t) = e^{t\mathcal{L}}u_0+ \int_{0}^{t}e^{(t-\tau)\mathcal{L}}f(\tau)u(\tau)\D\tau,
\end{eqnarray}
where  $u_0$ and $f(\tau), \ u(\tau)$ for  fixed $\tau$ are elements of the appropriate Banach spaces.
Assumption \ref{assumption1} guarantees that differential operator $\mathcal{L}$  is the infinitesimal generator of  a strongly continuous semigroup $\{\ee^{t\mathcal{L}}\}_ {t\geq 0}$ on $L^2(\Omega)$ and therefore $\max_{t \in [0,t^\star]}\|\ee^{t\mathcal{L}}\|_{L^2(\Omega)\leftarrow L^2(\Omega)}\leq C(t^\star)$, where $C(t^\star)$ is some constant independent of $t$
\cite{PAZY,EVANS}.  
Moreover,  
since the following a priori estimate hold
\begin{eqnarray}\label{apriori}
\|u\|_{H^{2p}(\Omega)}\leq C \left(\|\mathcal{L}u\|_{L^2(\Omega)}+\|u\|_{L^2(\Omega)}\right),
\end{eqnarray}
for any $u\in D(\mathcal{L})$, where $C>0$ is a constant \cite{agmon,PAZY}, operator $\ee^{t\mathcal{L}} $ is also bounded in norm $\| \ \ \|$ of space $D(\mathcal{L})$ for any $t\in [0,t^\star]$. Indeed, by using (\ref{apriori}) we have
\begin{eqnarray*}
	\|\ee^{t\mathcal{L}}\|_{D(\mathcal{L})\leftarrow D(\mathcal{L})}=\sup_{\|u\|\leq 1}\|\ee^{t\mathcal{L}}u\|&\leq& C \sup_{\|u\|\leq 1}\left(\|\mathcal{L}\ee^{t\mathcal{L}}u\|_{L^{2}(\Omega)}
	+\|\ee^{t\mathcal{L}}u\|_{L^{2}(\Omega)}\right)\\
	&=&C \sup_{\|u\|\leq 1}\left(\|\ee^{t\mathcal{L}}\mathcal{L}u\|_{L^{2}(\Omega)}
	+\|\ee^{t\mathcal{L}}u\|_{L^{2}(\Omega)}\right)\\
	&\leq&C \|\ee^{t\mathcal{L}}\|_{L^2(\Omega)\leftarrow L^2(\Omega)}\sup_{\|u\|\leq 1}\left(\|\mathcal{L}u\|_{L^{2}(\Omega)}
	+\|u\|_{L^{2}(\Omega)}\right)\\
 &\leq& C \|\ee^{t\mathcal{L}}\|_{L^2(\Omega)\leftarrow L^2(\Omega)}\sup_{\|u\|\leq 1}\left(C_1\|u\|
+\|u\|_{L^{2}(\Omega)}\right)\\
	&\leq& C(t^\star)
\end{eqnarray*}
and again constant $C(t^\star)$ depends on coefficients  $a_{\bm{p}}(x)$ of $\mathcal{L}$, but is independent of time variable $t$.

We introduce  the sequence $\left\{u^{[n]}\right\}_{n=0}^\infty \subset  C\left([0,t^\star];D(\mathcal{L})\right)=:V $ 
\begin{eqnarray} \label{partial_NS}
u^{[n]}(t)=\sum_{d=0}^{n }T^d \ee^{t\mathcal{L}}u_0,\quad u^{[0]}(t)=\ee^{t\mathcal{L}}u_0, \qquad t \in [0,t^\star].
\end{eqnarray}
Linear operator $T:V\rightarrow V$ is determined by
\begin{equation}\label{T}
Tu(t) = \int_{0}^{t}\ee^{(t-\tau)\mathcal{L}}f(\tau)u(\tau) \D\tau, \qquad t \in [0,t^\star],
\end{equation}
where $\mathcal{L}$   is the infinitesimal generator of  a strongly continuous semigroup,  and function $f\in C\left([0,t^\star];H^{2p}(\Omega)\right)$.
We  show that  sequence (\ref{partial_NS})  is convergent to the solution $u \in C^1([0,t^\star]; L^2(\Omega)) \cap C([0,t^\star];D(\mathcal{L}))$  of equation (\ref{eq:1.5}). 
We use the following estimate for the Sobolev norm
\begin{align}\label{Sobolev_estimation}
\|hg\|_{H^{2p}(\Omega)} \leq M 	\|h\|_{H^{2p}(\Omega)} 	\| g\|_{H^{2p}(\Omega)}, \qquad \Omega \subset \mathbb{R}^m, \qquad 2p>m/2,
\end{align}
for $h, g \in   H^{2p}(\Omega) $, where  $M$ depends only on $p$ and $m$.
In the beginning, we show that expression $T^d u(t)$
\begin{align*}\label{nested_int}
T^d u(t) = \int_{0}^{t}\ee^{(t-\tau_1)\mathcal{L}}f(\tau_1)\int_{0}^{\tau_1}\ee^{(\tau_1-\tau_2)\mathcal{L}}f(\tau_2)\dots\int_{0}^{\tau_{d-1}}\ee^{(\tau_{d-1}-\tau_d)\mathcal{L}}f(\tau_d)u(\tau_d)\D \tau_{d} \dots \D \tau_1,
\end{align*}
is uniformly bounded in norm $\| \ \ \|$ of space $D(\mathcal{L})$ by a constant independent of $t$.
\begin{lemma} \label{T^d_bounded}
	Suppose that Assumption \ref{assumption1} is satisfied. Let $u \in V$ ($V$ is the domain of operator $T$ defined in  (\ref{T})) and  $2p>m/2$.
	Then there exist a constant $M$ (depending only on $p$ and  $m$),  such that
	\begin{eqnarray}
	\label{norm_estimation}
	\|T^d u(t) \|  \leq M^d C_1^d C_2^dC_3 \frac{(t^\star)^  d}{d!}, \quad t\in[0,t^\star],
	\end{eqnarray} 
	where  $C_1:= \max_{t \in [0,t^\star]}\|\ee^{t\mathcal{L}}\|_{D(\mathcal{L})\leftarrow D(\mathcal{L})} $, $C_2:= \max_{t \in [0,t^\star]}\|f(t)\|$ and $C_3:= \max_{t \in [0,t^\star]}\|u(t)\|$. 
\end{lemma}
\begin{proof}
	By applying inequality  (\ref{Sobolev_estimation}), and using basic properties of the operator norm, we can  estimate  the term $\|T^d u(t) \|$
	\begin{eqnarray*}
		\|T^d u(t)\|&=&\left\| \int_{0}^{t}\ee^{(t-\tau_1)\mathcal{L}}f(\tau_1)\int_{0}^{\tau_1}\ee^{(\tau_1-\tau_2)\mathcal{L}}f(\tau_2)\dots\int_{0}^{\tau_{d-1}}\ee^{(\tau_{d-1}-\tau_d)\mathcal{L}}f(\tau_d)u(\tau_d)\D \tau_{d} \dots \D \tau_1\right\|\\
		&=&
		\left\| \int_{0}^{t}\int_{0}^{\tau_1}\dots\int_{0}^{\tau_{d-1}}\ee^{(t-\tau_1)\mathcal{L}}f(\tau_1)\ee^{(\tau_1-\tau_2)\mathcal{L}}f(\tau_2)\dots\ee^{(\tau_{d-1}-\tau_d)\mathcal{L}}f(\tau_d)u(\tau_d)\D \tau_{d} \dots \D \tau_1\right\|\\
		& \leq& \int_{0}^{t}\int_{0}^{\tau_1}\dots\int_{0}^{\tau_{d-1}}\left\|\ee^{(t-\tau_1)\mathcal{L}}f(\tau_1)\ee^{(\tau_1-\tau_2)\mathcal{L}}f(\tau_2)\dots\ee^{(\tau_{d-1}-\tau_d)\mathcal{L}}f(\tau_d)u(\tau_d)\right\|\D \tau_{d} \dots \D \tau_1\\
		&\leq& \int_{0}^{t}\int_{0}^{\tau_1}\dots\int_{0}^{\tau_{d-1}}\left\|\ee^{(t-\tau_1)\mathcal{L}}\right\|\left\|f(\tau_1)\ee^{(\tau_1-\tau_2)\mathcal{L}}f(\tau_2)\dots\ee^{(\tau_{d-1}-\tau_d)\mathcal{L}}f(\tau_d)u(\tau_d)\right\|\D \tau_{d} \dots \D \tau_1\\
		&\leq& C_1C_2M\int_{0}^{t}\int_{0}^{\tau_1}\dots\int_{0}^{\tau_{d-1}} \left\|\ee^{(\tau_1-\tau_2)\mathcal{L}}f(\tau_2)\dots\ee^{(\tau_{d-1}-\tau_d)\mathcal{L}}f(\tau_d)u(\tau_d)\right\|\D \tau_{d} \dots \D \tau_1=: (\bm{\star})\\
	\end{eqnarray*}
	Constant $M$ comes from inequality (\ref{Sobolev_estimation}). Proceeding analogously and repeatedly we obtain 
	\begin{eqnarray*}
		(\bm{\star})&\leq&  M^d C_1^d C_2^dC_3\int_{0}^{t}\int_{0}^{\tau_1}\dots\int_{0}^{\tau_{d-1}}\D \tau_{d} \dots \D \tau_1  \leq  M^d C_1^d C_2^dC_3 \frac{(t^\star)^  d}{d!}, \quad \forall t\in[0,t^\star],
	\end{eqnarray*}
	and the right hand side of the above inequality is independent of $t$, which completes the proof.
\end{proof}
\begin{theorem} \label{theorem_Neumann_Series}
	Suppose that Assumption \ref{assumption1} is satisfied,  $2p>m/2$ and $t \in [0,t^\star]$. Then series  $$\sum_{d=0}^{\infty }T^d \ee^{t\mathcal{L}}u_0, \quad t \in [0,t^\star],$$
	converges absolutely and uniformly to a function $u^\star \in C^1([0,t^\star];L^2(\Omega)) \cap C([0,t^\star];D(\mathcal{L}))$ and  $u^\star$ is the solution of the equation (\ref{eq:1.5}).
\end{theorem}
\begin{proof}
	Let us  denote $C_1:= \max_{t \in [0,t^\star]}\|\ee^{t\mathcal{L}}\|_{D(\mathcal{L})\leftarrow D(\mathcal{L})}$, $C_2:= \max_{t \in [0,t^\star]}\|f(t)\|$. By using (\ref{norm_estimation}), we estimate 
	$$\max_{t \in [0,t^\star]}\|T^d \ee^{t\mathcal{L}}u_0\|\leq\underbrace{M^d C_1^{d+1} C_2^d \|u_0\|\frac{(t^\star)^d}{d!}}_{=:A_d}, \quad \forall \ d \in \mathbb{N}.$$
	Since $\sum_{d=0}^{\infty}A_d =   C_1 \|u_0\|\exp(t^\star M C_1  C_2) $ 
	and space $D(\mathcal{L})$ is a Banach space, the Weierstrass M-test ensures that the  sequence $\left\{u^{[n]}\right\}_{n=0}^\infty$ is convergent absolutely and uniformly to some function $u^\star \in C([0,t^\star],D(\mathcal{L})) $.
	Thus, for each $t \in [0,t^\star]$ we have
	\begin{eqnarray}\label{u_star}
	u^{\star}(t):=\lim_{n \to \infty } u^{[n]}(t)=\sum_{d=0}^{\infty}T ^d \ee^{t\mathcal{L}}u_0,
	\end{eqnarray}
	where $u^{[n]}$ is defined by (\ref{partial_NS}).
	The operator $I-T$ is a bijection and function $u^{\star}$ is the solution of (\ref{eq:Duhamel}).
	Indeed, it is easy to observe that 
	$$(I-T)\sum_{d=0}^{n}T^d= \sum_{d=0}^{n}T^d(I-T)= I-T^{n+1}.$$
	Since $\|T^{n+1}\|\rightarrow 0$ as $ n \rightarrow \infty$,  for each $t \in [0,t^\star]$ we obtain
	$$(I-T)u^\star(t) = \ee^{t\mathcal{L}}u_0,$$
	which is the equivalent form of  equation (\ref{eq:Duhamel}). 
	Moreover, one can observe that  the mapping $t \mapsto u^{\star}(t)$ is differentiable for each $t>0$, when $u_0 \in D(\mathcal{L})$ and $f \in C([0,t^\star],H^{2p}(\Omega)) $. Indeed, by considering the mild formulation of  (\ref{eq:1.5}), namely
	\begin{eqnarray}\label{Mild}
	u^\star(t) = \ee^{t\mathcal{L}}u_0+ \int_{0}^{t}\ee^{(t-\tau)\mathcal{L}}f(\tau)u^\star(\tau)\D\tau,
	\end{eqnarray}
	of course  the first part $\ee^{t\mathcal{L}}u_0$ is continuously differentiable, $\frac{d}{dt}\ee^{t\mathcal{L}}u_0 = \mathcal{L}\ee^{t\mathcal{L}}u_0, \ \forall t>0$, and since $u^\star(\tau)$ is continuous, function $g(t)$, defined as $g(t):=\int_{0}^{t}\ee^{(t-\tau)\mathcal{L}}f(\tau)u^\star(\tau)\D\tau$, is also continuously differentiable (see \cite{Renardy},  Theorem 12.16). Since $\mathcal{L}$ is a closed operator, we have
	\begin{equation}\label{pochodna_wyr}
	\frac{d}{dt}u^\star(t)  =\ee^{t\mathcal{L}}\mathcal{L}u_0+ f(t)u^\star(t)+\int_{0}^{t}\ee^{(t-\tau)\mathcal{L}}\mathcal{L}f(\tau)u^\star(\tau)\D\tau.
	\end{equation}
	Since $f(t) \in H^{2p}(\Omega)$ and $u^\star(t) \in H^{2p}(\Omega)\cap H_0^p(\Omega)$, $\forall t\in[0,t^\star]$,  it follows that $f(t)u^\star(t) \in H^{2p}(\Omega)\cap H_0^p(\Omega)$ = $D(\mathcal{L})$. Therefore,  each expression on the right-hand side of (\ref{pochodna_wyr}) is well-defined.
	We have shown that the solution  $u^\star $ of (\ref{Mild})  belongs to space  $ C^1([0,t^\star]; L^2(\Omega)) \cap C([0,t^\star];D(\mathcal{L}))$. Now it is obvious that $u^\star$ is the solution of (\ref{eq:1.5}).
\end{proof}
The above convergence proof of the sequence (\ref{partial_NS}) is similar to that presented in \cite{KKRP}. However, to show the convergence of (\ref{partial_NS}) to the solution of equation (\ref{eq:1.5}), one needs to provide the convergence of the sequence (\ref{partial_NS}) in the appropriate Sobolev norm.

Needless to say, the existence and uniqueness of linear evolution equations is a well-known and well-established theory in mathematics. Theorem \ref{theorem_Neumann_Series} is only needed to show that solution to the equation can be presented as a series of multivariate integrals. We utilize this fact  in the numerical approximation a highly oscillatory solution.  
It is easy to notice that in the proof of Theorem \ref{theorem_Neumann_Series}, operator $T$ need not be a contraction mapping, but the Neumann series is still convergent. It can be shown that for arbitrarily time variable $t>0$, there exists number $s$, such that for any $d>s$, operator $T^d$ is a contraction mapping, i.e. $\|T^d\|<1$.

\begin{remark}
	The  mild formulation of equation (\ref{eq:1.5}) allows us  to relax the assumptions on the initial condition $u_0$. Instead of requiring $u_0 \in D(\mathcal{L})$, let us assume  for a moment that $u_0\in L^2(\Omega)$ and additionally that $f$ is a bounded function $\|f\|_\infty <\infty$. By following a similar approach to the proof of Theorem \ref{theorem_Neumann_Series}, one can also show the convergence of the sequence (\ref{partial_NS}) in the norm of the space $C([0,t^\star];L^2(\Omega))$ to the mild solution $u^\star$ of equation (\ref{eq:1.5}). In this case, the additional assumption $2p > m/2$ becomes redundant.
\end{remark}
The series (\ref{u_star}) is known as the Neumann series.
Equation (\ref{eq:Duhamel}) can be expressed symbolically as
$$u(t) = \ee^{t\mathcal{L}}u_0+ Tu(t).$$
It can be easy verified that for each $n$, the term $u^{[n]}(t)$ satisfies the relation
$$u^{[n]}(t) = \ee^{t\mathcal{L}}u_0+Tu^{[n-1]}(t), \quad u^{[0]}(t) = \ee^{t\mathcal{L}}u_0.$$
Therefore, in the remaining text, we will use the terms `$n$-th partial sum of the Neumann series' and `$n$-th iteration of the equation' interchangeably to refer to the expression $u^{[n]}(t)$.

\section{Presenting terms of the Neumann series as a  sum of multivariate integrals}\label{sec3}
According to Theorem \ref{theorem_Neumann_Series}, the solution to equation (\ref{eq:1.5}) can be expressed as the Neumann series,   $u(t) =\sum_{d=0}^{\infty}T^d\ee^{t\mathcal{L}}u_0$.
In this section, we show that  for function $f$, defined in (\ref{ideal_function_f}), each term  $T^d\ee^{t\mathcal{L}}, \  d=1,2,\dots$ of the Neumann series   is a sum of multivariate highly oscillatory integrals. Subsequently, we will prove that each $T^d\ee^{t\mathcal{L}}$ can be approximated by a partial sum of the asymptotic expansion.
By linearity of semigroup operator $\{\ee^{t\mathcal{L}}\}_{t\geq 0}$,  we can convert expression $T^d\ee^{t\mathcal{L}}$ into a more convenient form,  for function $f$ defined in (\ref{ideal_function_f})
\footnotesize
\begin{eqnarray}\label{sum_of integrals}
&&T^d \ee^{t\mathcal{L}} =  \nonumber
\nonumber \int_{0}^{t}\int_{0}^{\tau_d}\dots\int_{0}^{\tau_{2}}\ee^{(t-\tau_d)\mathcal{L}}f(\tau_d)\ee^{(\tau_d-\tau_{d-1})\mathcal{L}}f(\tau_{d-1})\dots \ee^{(\tau_{2}-\tau_1)\mathcal{L}}f(\tau_1)\ee^{\tau_1\mathcal{L}}\D \tau_{1} \dots \D \tau_d=\\ \nonumber
&& \sum_{\substack{1\leq n_1,\dots n_d\leq N}}\int_{0}^t\int_{0}^{\tau_d}\dots\int_{0}^{\tau_2}\ee^{(t-\tau_d)\mathcal{L}}\alpha_{n_d}\ee^{(\tau_d-\tau_{d-1})\mathcal{L}}\alpha_{n_{d-1}}\dots \ee^{(\tau_{2}-\tau_1)\mathcal{L}}\alpha_{n_1}\ee^{\tau_1\mathcal{L}}\ee^{\ii\omega(\tau_1 n_1+\dots+\tau_{d}n_{d})}\D \tau_1\dots\D\tau_d=\\
&& \sum_{\bm{n}\in \{1,\dots,N\}^d}\int_{\sigma_d(t)}F_{\bm{n}}(t,\bm{\tau})\ee^{\ii\omega  \bm{n}^T\bm{\tau} }\D \bm{\tau} = \sum_{\bm{n}\in \{1,\dots,N\}^d} I[F_{\bm{n}},\sigma_d(t)], \nonumber
\end{eqnarray}
\normalsize
where
\begin{eqnarray} \label{function_f_operator}
\bm{\tau} &=&  (\tau_1,\tau_2,\dots,\tau_d), \nonumber\\
F_{\bm{n}}(t,\bm{\tau}) &=& \ee^{(t-\tau_d)\mathcal{L}}\alpha_{n_d}\ee^{(\tau_d-\tau_{d-1})\mathcal{L}}\alpha_{n_{d-1}}\dots \ee^{(\tau_{2}-\tau_1)\mathcal{L}}\alpha_{n_1}\ee^{\tau_1\mathcal{L}},\\
I[F_{\bm{n}},\sigma_d(t)]&=&  \int_{\sigma_d(t)}F_{\bm{n}}(t,\bm{\tau})\ee^{\ii\omega \bm{n}^T \bm{\tau} }\D \bm{\tau} \label{I[F,S_d(t)]}
\end{eqnarray}
and $\sigma_d(t)$ denotes a $d$-dimensional simplex
$$\sigma_d(t)=\{\bm{\tau}:=(\tau_1,\tau_2,\dots,\tau_d)\in \mathbb{R}^d:t\geq \tau_d\geq \tau_{d-1}\geq \dots \geq \tau_2\geq \tau_1 \geq 0\}.$$
Using the above notation,  solution $u(t)$ of (\ref{eq:Duhamel}) can be written as
$$u(t) = \ee^{t\mathcal{L}}u_0+\sum_{d=1}^{\infty}\sum_{\bm{n}\in \{1,\dots,N\}^d} I[F_{\bm{n}},\sigma_d(t)]u_0.$$
Each $F_{\bm{n}}(t,\bm{\tau})$ is  a linear operator, $F_{\bm{n}}(t,\bm{\tau}): \mathcal{D}(\mathcal{L})\rightarrow \mathcal{D}(\mathcal{L})$ and  $\bm{n}\in \mathbb{N}^d$ is a vector corresponding to  $F_{\bm{n}}$ in the sense that $\bm{n}$ appears in the frequency exponent of integral (\ref{I[F,S_d(t)]}).


In our approach, to expand asymptotically multivariate highly oscillatory integral of type (\ref{I[F,S_d(t)]}), we will exploit repeatedly Fubini's theorem, properties of the semigroup operator and integration by parts. 
Therefore  one should be able to compute successive partial derivatives of  the expression of type (\ref{function_f_operator}). One may prove by induction the following 
\begin{lemma}\label{lemma_ady}
	Assume $\alpha \in D(\mathcal{L}^k)$, where $D(\mathcal{L}^k):=\{u\in D(\mathcal{L}^{k-1}): \mathcal{L}^{k-1}u \in D(\mathcal{L})\}$, $k=1,2,\dots$ Then
	\begin{eqnarray*}\label{equality_ady}
		&&	\partial_\tau^k\left(\ee^{(t-\tau)\mathcal{L}}\alpha\ee^{\tau\mathcal{L}}\right)= \ee^{(t-\tau)\mathcal{L}}\underbrace{\Big[\dots\big[[\alpha,\mathcal{L}],\mathcal{L}\big],\dots\Big]}_{\mathcal{L} \ appears\ k \ \text{times}}\ee^{\tau\mathcal{L}}=(-1)^k\ee^{(t-\tau)\mathcal{L}}ad_\mathcal{L}^k\big(\alpha\big)\ee^{\tau\mathcal{L}},
	\end{eqnarray*}
	where $ad_\mathcal{L}^0\big(\alpha\big)=\alpha, \quad  ad_\mathcal{L}^{k}\big(\alpha\big)=\big[\mathcal{L}, ad_\mathcal{L}^{k-1}\big(\alpha\big)\big]$ and $[X,Y]\equiv XY-YX$ is the commutator of $X$ and $Y$.
\end{lemma}
Unless otherwise stated, we  assume that each appearing expression of type $\mathcal{L}^k\alpha_{n}$ is well defined for  any functions $\alpha_{n}, n=1,\dots,N$.

In the asymptotic expansion of integral of type (\ref{I[F,S_d(t)]}), the number of terms in a partial sum of the asymptotic series grows exponentially as the dimension $d$ of the domain  increases, therefore to facilitate the notation we introduce the following definitions.

\begin{defi}
	By $\bm{v}_\ell^d \in \{0,1\}^d$, $\ell=0,1,\dots d$ we denote the vertices of simplex $\sigma_d(1)$
	\begin{eqnarray*}
		\bm{v}_0^d &=&(1,1,1,\dots,1,1);  \\
		\bm{v}_{1}^d &=&(0,1,1,\dots,1,1) ;\\
		&\vdots&\\
		\bm{v}_{\ell}^d &=&(\underbrace{0,\dots,0}_{\ell \ zeros  },1,\dots,1).
	\end{eqnarray*}
\end{defi}
\begin{defi}\label{set_of_vectors}
	Let 	$\Phi_\ell ^d, \ \ell=0,1,\dots,d$, be a family of points such that
	$$\Phi_0 ^d= \{(1,1,\dots,1)\},$$
	$$\Phi_\ell ^d= \{\bm{\phi} \in \{0,1\}^d: \bm{\phi} = (\phi_1,\phi_2,\dots,\phi_d), \  \phi_\ell=0, \ \phi_j=1, \ \text{for} \ j>\ell \}$$
	In other words, $\bm{\phi}\in \Phi_\ell ^d$ if the last zero is at $\ell$-th coordinate of $\bm{\phi}$, next coordinates are ones only.
\end{defi}
\begin{defi}\label{sequence_partial_derivative} 
	For $\bm{\phi}= (\phi_1,\phi_2,\dots,\phi_d)\in \{0,1\}^d$ and multi-index $\bm{k}=(k_1,\dots,k_d)$, we define the sequence of partial derivatives of $F(t,\tau_1,\tau_2,\dots,\tau_d)$
	\begin{eqnarray*}
		&&F^{\bm{k}}[\bm{\phi}](t)=\partial_{\tau_d}^{k_d}\left(\dots \partial_{\tau_2}^{k_2}\left(\partial_{\tau_1}^{k_1}F(t,\tau_1,\tau_2,\dots,\tau_d)\,\rule[-4pt]{0.55pt}{17pt}_{\,\tau_1=\phi_1\tau_{2}}\right)\rule[-4pt]{0.55pt}{17pt}_{\,\tau_{2}=\phi_2\tau_{3}}\dots\right)\rule[-4pt]{0.55pt}{17pt}_{\,\tau_{d}=\phi_d t}.
	\end{eqnarray*}
\end{defi}

Definition \ref{sequence_partial_derivative} we understand as follows: 
first, we compute the $k_1$-th partial derivative for variable $\tau_1$. Then we substitute in place of $\tau_1$ variable   $\tau_2$ or $0$. Subsequently, we compute the $k_2$-th partial derivative with respect to variable $\tau_2$ and then again we substitute $\tau_2=\tau_3$ or $\tau_2=0$. We proceed in this manner until we compute the $k_d$-th partial derivative with respect to  $\tau_d$ and substitute $\tau_d=\phi_dt$, where $\phi_d =0$ or $\phi_d =1$. 
\begin{exam}
	To present how Definition \ref{sequence_partial_derivative} operates for $F_{\bm{n}}$ definied in (\ref{function_f_operator}), we compute the first few expressions $F^{\bm{k}}_{\bm{n}}[\bm{\phi}](t)$ for  $d=1,2$. We assume that  $\alpha_{n}, \ n=1,\dots,N$ are  sufficiently smooth functions.
	\begin{eqnarray*}
		F^{k_1}_{n_1}[1](t)&=& \partial_{\tau_1}^{k_1}\left(\ee^{(t-\tau_1)\mathcal{L}}\alpha_{n_1}\ee^{\tau_1\mathcal{L}}\right)\rule[-4pt]{0.7pt}{17pt}_{\,\tau_1=t}=(-1)^{k_1}ad_\mathcal{L}^{k_1}\big(\alpha_{n_1}\big)\ee^{t\mathcal{L}},\\
		F_{n_1}^{k_1}[0](t)&=& \partial_{\tau_1}^{k_1}\left(\ee^{(t-\tau_1)\mathcal{L}}\alpha_{n_1}\ee^{\tau_1\mathcal{L}}\right)\rule[-4pt]{0.7pt}{17pt}_{\,\tau_1=0}=(-1)^{k_1}\ee^{t\mathcal{L}}ad_\mathcal{L}^{k_1}\big(\alpha_{n_1}\big),\\
		F^{\bm{k}}_{\bm{n}}[(1,1)](t)&=&\partial_{\tau_2}^{k_2}\left(\partial_{\tau_1}^{k_1}\left(\ee^{(t-\tau_2)\mathcal{L}}\alpha_{n_2}\ee^{(\tau_2-\tau_1)\mathcal{L}}\alpha_{n_1}\ee^{\tau_1\mathcal{L}}\right)\rule[-4pt]{0.7pt}{17pt}_{\,\tau_1=\tau_2}\right)\rule[-4pt]{0.7pt}{17pt}_{\,\tau_2=t} =(-1)^{|\bm{k}|} ad^{k_2}_{\mathcal{L}}\left(\alpha_{n_2} ad_{\mathcal{L}}^{k_1}(\alpha_{n_1})\right)\ee^{t\mathcal{L}},\\
		F^{\bm{k}}_{\bm{n}}[(1,0)](t)&=&\partial_{\tau_2}^{k_2}\left(\partial_{\tau_1}^{k_1}\left(\ee^{(t-\tau_2)\mathcal{L}}\alpha_{n_2}\ee^{(\tau_2-\tau_1)\mathcal{L}}\alpha_{n_1}\ee^{\tau_1\mathcal{L}}\right)\rule[-4pt]{0.7pt}{17pt}_{\,\tau_1=\tau_2}\right)\rule[-4pt]{0.7pt}{17pt}_{\,\tau_2=0} =(-1)^{|\bm{k}|}  \ee^{t\mathcal{L}}ad^{k_2}_{\mathcal{L}}\left(\alpha_{n_2} ad_{\mathcal{L}}^{k_1}(\alpha_{n_1})\right),\\
		F^{\bm{k}}_{\bm{n}}[(0,1)](t)&=&\partial_{\tau_2}^{k_2}\left(\partial_{\tau_1}^{k_1}\left(\ee^{(t-\tau_2)\mathcal{L}}\alpha_{n_2}\ee^{(\tau_2-\tau_1)\mathcal{L}}\alpha_{n_1}\ee^{\tau_1\mathcal{L}}\right)\rule[-4pt]{0.7pt}{17pt}_{\,\tau_1=0}\right)\rule[-4pt]{0.7pt}{17pt}_{\,\tau_2=t} =(-1)^{|\bm{k}|} ad_\mathcal{L}^{k_2}\big(\alpha_{n_2}\big)\ee^{t\mathcal{L}}ad_\mathcal{L}^{k_1}\big(\alpha_{n_1}\big),\\
		F^{\bm{k}}_{\bm{n}}[(0,0)](t)&=&\partial_{\tau_2}^{k_2}\left(\partial_{\tau_1}^{k_1}\left(\ee^{(t-\tau_2)\mathcal{L}}\alpha_{n_2}\ee^{(\tau_2-\tau_1)\mathcal{L}}\alpha_{n_1}\ee^{\tau_1\mathcal{L}}\right)\rule[-4pt]{0.7pt}{17pt}_{\,\tau_1=0}\right)\rule[-4pt]{0.7pt}{17pt}_{\,\tau_2=0} =(-1)^{|\bm{k}|} \ee^{t\mathcal{L}}ad_\mathcal{L}^{k_2}\big(\alpha_{n_2}\big)ad_\mathcal{L}^{k_1}\big(\alpha_{n_1}\big).
	\end{eqnarray*}
\end{exam}

It can be shown that expression $F_{\bm{n}}^{\bm{k}}[\bm{\phi}]$, for operator $F_{\bm{n}}$  defined in (\ref{function_f_operator}), takes the following form 
\begin{eqnarray*}\label{alfa_L_sequence}
	F_{\bm{n}}^{\bm{k}}[\bm{\phi}](t) = \sum_{\bm{m}\leq \bm{k}}^{}{\bm{k}\choose \bm{m}}(-1)^{|\bm{m}|+|\bm{k}|}\mathcal{L}^{r_{d+1}} \alpha_{n_d}\mathcal{L}^{r_{d}}\alpha_{n_{d-1}}\dots \mathcal{L}^{r_{\ell+1}}\alpha_{n_{\ell+1}}\ee^{t\mathcal{L}}\mathcal{L}^{r_{\ell}}\alpha_{n_{\ell}}\dots \mathcal{L}^{r_2}\alpha_{n_1}\mathcal{L}^{r_1},
\end{eqnarray*}
for certain  vector $\bm{\phi}\in \Phi^d_\ell$, sufficiently smooth functions $\alpha_{n}, \ n=1,\dots,N$, multi-indexes $\bm{k}=(k_1,\dots.k_d), \ \bm{m}=(m_1,\dots,m_d)$ and ${\bm{k}\choose \bm{m}} = \frac{\bm{k}!}{\bm{m}!(\bm{k}-\bm{m})!}$. 
Numbers $r_1,\dots r_{d+1}$ satisfy $r_1+\dots+r_{d+1}=|\bm{k}|$ and $\bm{m}\leq \bm{k}$ means $m_i\leq k_i$ for $i=1,2,\dots ,d$.

\section{Asymptotic expansion of a highly oscillatory integral subject to nonresonance  condition}\label{sec4}
Let $\bm{n}=(n_1,\dots,n_d) \in \mathbb{N}^d$. In this section, we assume that the coordinates of $\bm{n}$ satisfy the nonresonance condition:
\begin{eqnarray}\label{sum_nonorthogonal_condition}
n_j + n_{j+1} + \dots + n_{r-1} + n_r \neq 0,
\end{eqnarray}
for each $1\leq j\leq r \leq d$. This condition implies that $\bm{n}$ is not orthogonal to the faces of the simplex $\sigma_d(t)$.
This case includes the situation where the highly oscillatory potential $f$ takes the form (\ref{ideal_function_f}).
This section shows that the integral $I[F_{\bm{n}},\sigma_d(t)]$ defined in (\ref{I[F,S_d(t)]}) can be expressed as a partial sum of the asymptotic series.
Let us start with the first term of the Neumann series. By using integration by parts, the integral $[F_{n_1},(0,t)]$ can be presented as the following sum:
\begin{equation}\label{IA0B}
\begin{split}
&\int_{0}^{t}\ee^{\ii\omega\tau n_1}\ee^{(t-\tau)\mathcal{L}}\alpha_{n_1}\ee^{\tau\mathcal{L}}u_0\D\tau =  \sum_{k=0}^{r-1}\frac{1}{(\ii  \omega n_1)^{k+1}}\Big[\ee^{\ii\omega t n_1} ad_\mathcal{L}^k\big(\alpha_{n_1}\big)\ee^{t\mathcal{L}}-\ee^{t\mathcal{L}}ad_\mathcal{L}^k\big(\alpha_{n_1}\big)\Big]u_0 \\ 
&\hspace{5cm}+\frac{1}{(\ii\omega n_1)^r}\int_{0}^{t}\ee^{\ii\omega\tau n_1}\ee^{(t-\tau)\mathcal{L}}ad_\mathcal{L}^r\big(\alpha_{n_1}\big)\ee^{\tau\mathcal{L}}u_0\D\tau,
\end{split}
\end{equation}
and the above partial sum  approximates the integral with error $\mathcal{O}(\omega^{-r-1})$.
The general formula for the integral $I[F_{\bm{n}},\sigma_d(t)]$ over a $d$-dimensional simplex is unfortunately much more complicated. However, equation (\ref{IA0B}) shows the main idea of our considerations. In the later derivation, to expand asymptotically multivariate highly oscillatory integral of type (3.2), we will use simple tools, including induction, Fubini's theorem and identity (\ref{IA0B}).

The following definition determines  coefficients that will appear as a result of integrating by parts.
\begin{defi}\label{numbers_A}
	For vector $ \Phi^d_\ell \ni \bm{\phi}= (\phi_1,\dots,\phi_d) = (\phi_1,\dots,\phi_{\ell-1},0,1,\dots,1)$ (i.e. the last zero on $\ell$-th coordinate, next coordinated are ones only), multi-index $\bm{k}=(k_1,\dots,k_{d})$,  and vector $\bm{n}=(n_1,\dots,n_d) \in\mathbb{N}^d$ satisfying condition (\ref{sum_nonorthogonal_condition}),  we introduce the following rational numbers  
	\small
	\begin{eqnarray}\label{expressions_A}
	&&A_{\bm{{k}}}[\bm{\phi}](\bm{n})=\\
	&&\frac{(-1)^{\phi_1+1}}{n_1^{k_1+1}} \frac{(-1)^{\phi_2+1}}{(n_1\phi_1+n_2)^{k_2+1}} \frac{(-1)^{\phi_3+1}}{\left((n_1\phi_1+n_2)\phi_2+n_3\right)^{k_3+1}}\dots \frac{(-1)^{\phi_d+1}}{\left(\dots\left((n_1\phi_1+n_2)\phi_2+n_3\right)\phi_3+\dots+n_d\right)^{k_d+1}} \nonumber
	\end{eqnarray}
	\normalsize
	Let us note that due to assumption  (\ref{sum_nonorthogonal_condition}),  we never divide  by zero in the above expressions.
	\end{defi}
\noindent One can observe that the numbers  $A_{\bm{{k}}}[\bm{\phi}](\bm{n})$ satisfies the following recursive relation
	\begin{eqnarray*}
	&&A_{k_1}[\phi_1](n_1) = \frac{(-1)^{\phi_1+1}}{n_1^{k_1+1}},\\
	&&A_{\bm{k}^{}}[(\bm{\phi},\phi_{d+1})]((\bm{n},n_{d+1}))=A_{\bm{{\tilde{k}}}}[\bm{\phi}](\bm{n}) \frac{(-1)^{\phi_{d+1}+1}}{[(\bm{v}_{\ell}^d,1)\cdot (\bm{n},n_{d+1})]^{k_{d+1}+1}},
\end{eqnarray*}
where  $\bm{\phi}\in \Phi^d_\ell$, (the last zero is at the $\ell$-th coordinate, and at coordinates $\ell+1,...,d$ there are ones), vertex $\bm{v}_\ell^d = (0,\dots,0,1,\dots,1)$ (the last zero also is at the $\ell$-th coordinate) of simplex $\sigma_d(1)$, multi-indexes $\bm{k}=(k_1,\dots,k_{d+1})$, $\bm{\tilde{k}}=(k_1,\dots,k_{d})$ and vector $\bm{n}=(n_1,\dots,n_d) \in\mathbb{N}^d$ satisfying condition (\ref{sum_nonorthogonal_condition}).
	Recursive definition will be needed later in the proof of Theorem \ref{theorem_asymptotic_non_resonances}. 

\begin{exam}
	For  $d=3$  we compute  coefficients $A_{\bm{{k}}}[\bm{\phi}](\bm{n})$ for different $\phi \in \Phi_\ell^3, \ \ell=0,1,2,3$.
	\begin{eqnarray*}
		&&\ell=0, \ A_{\bm{{k}}}[(1,1,1)](\bm{n})=   \frac{1}{n_1^{k_1+1}(n_1+n_2)^{k_2+1} (n_1+n_2+n_3)^{k_3+1}},\\
		&&\ell=1, \ A_{\bm{{k}}}[(0,1,1)](\bm{n})=  \frac{-1}{n_1^{k_1+1} (n_2)^{k_2+1} (n_2+n_3)^{k_3+1}},\\
		&&\ell=2, \  A_{\bm{{k}}}[(0,0,1)](\bm{n})= \frac{1}{n_1^{k_1+1}  n_2^{k_2+1} n_3^{k_3+1}},\\
		&&\ell=2, \  A_{\bm{{k}}}[(1,0,1)](\bm{n})=  \frac{-1}{n_1^{k_1+1} (n_1+n_2)^{k_2+1}n_3^{k_3+1}},\\
		&&\ell=3, \  A_{\bm{{k}}}[(0,0,0)](\bm{n})= \frac{-1}{n_1^{k_1+1} n_2^{k_2+1} n_3^{k_3+1}},\\
		&&\ell=3, \  A_{\bm{{k}}}[(1,1,0)](\bm{n})= \frac{-1}{n_1^{k_1+1} (n_1+n_2)^{k_2+1} (n_1+n_2+n_3)^{k_3+1}},\\
		&&\ell=3, \  A_{\bm{{k}}}[(0,1,0)](\bm{n})= \frac{1}{n_1^{k_1+1} n_2^{k_2+1} (n_2+n_3)^{k_3+1}},\\
		&&\ell=3, \  A_{\bm{{k}}}[(1,0,0)](\bm{n})=  \frac{1}{n_1^{k_1+1} (n_1+n_2)^{k_2+1} n_3^{k_3+1}}.
	\end{eqnarray*}
\end{exam}
\begin{remark}
	In the subsequent part of the paper, we employ the following notation: if $\bm{n}=(n_1,\dots,n_m)\in \mathbb{Z}^m$ is an $m$-dimensional vector, then $\bm{\tilde{n}}=(n_1,n_2,\dots,n_{m-1})\in \mathbb{Z}^{m-1}$ represents a $(m-1)$-dimensional vector obtained by excluding the last coordinate $n_m$ from $\bm{n}$.
\end{remark}
To simplify  complex  notation in the proof Theorem \ref{theorem_asymptotic_non_resonances}, we will also write $F:=F_{\bm{n}}$ for the operator defined in (\ref{function_f_operator}), if vector $\bm{n}$  corresponding  to $F$ is clear from the context. 
\begin{theorem}\label{theorem_asymptotic_non_resonances}
	Suppose that Assumption 1 is satisfied. In addition, suppose that $u_0, \ \alpha_n \in H^{2p(r+1)}(\Omega), \ n=1,\dots, N$. Let $F$ be the operator defined in (\ref{function_f_operator}),  and let $\bm{n}$ be the vector corresponding to $F$ that satisfies the nonresonance condition (\ref{sum_nonorthogonal_condition}). Integral (\ref{I[F,S_d(t)]}) can be expressed as the $r$-partial sum  $\mathcal{S}_r^{(d)}(t)$  of the asymptotic series with error  $E_r^{(d)}(t)$
	\begin{eqnarray}\label{series+error}
	I[F,\sigma_d(t)]=	\int_{\sigma_d(t)}F(t,\bm{\tau})\ee^{\ii\omega \bm{n}^T\bm{\tau} }\D \bm{\tau} = \mathcal{S}^{(d)}_r(t)+E^{(d)}_r(t),
	\end{eqnarray}
	where
	\begin{equation}\label{series_S_expansion}
	\mathcal{S}^{(d)}_r(t) = \sum_{|\bm{k}|=0}^{r-d}\frac{(-1)^{|\bm{k}|}}{(\ii\omega)^{d+|\bm{k}|}}\sum_{\ell=0}^{d}\ee^{\ii\omega t \bm{n}^T\bm{v}_{\ell}^d}\sum_{\bm{\phi}\in \Phi^d_{\ell}}^{}A_{\bm{k}}[\bm{\phi}](\bm{n})F^{\bm{k}}[\bm{\phi}](t), \quad r\geq d,
	\end{equation}
	and error $E_r^{(d)}(t)$ of the expansion is in recursive form
	\footnotesize
	\begin{eqnarray}\label{error_expansion}
	E^{(1)}_r(t)&=&\frac{(-1)^r}{(\ii\omega)^r}\frac{1}{n_1^{r}}\int_{0}^{t}\ee^{\ii\omega\tau_1 n_1}\partial_{\tau_1}^{r}F(t,\tau_1)\D\tau_1,\nonumber\\
	E_r^{(d)}(t)&=&\frac{(-1)^{r-d+1}}{(\ii\omega)^r}\sum_{|\bm{k}|=r-d+1}^{}\sum_{\ell=0}^{d-1}\sum_{\bm{\phi}\in \nonumber \Phi^{d-1}_{\ell}}\frac{1}{\left(\bm{n}^T\bm{v}^d_{\ell}\right)^{k_d}}A_{\bm{\tilde{k}}}[\bm{\phi}](\bm{\tilde{n}}) \int_{0}^{t}\ee^{\ii\omega \tau_d \bm{n}^T\bm{v}_{\ell}^d}\partial_{\tau_d}^{k_d}\left(\ee^{(t-\tau_d)\mathcal{L}}\alpha_{n_d}F^{\bm{\tilde{k}}}[\bm{\phi}](\tau_d)	\right)\D\tau_d\\
	&+&\int_{0}^{t}\ee^{(t-\tau_d)\mathcal{L}}\alpha_{n_d}E_r^{(d-1)}(\tau_d)\ee^{\ii\omega n_d\tau_d}\D\tau_d, \ \text{for} \ d\geq 2. 
	\end{eqnarray}
\end{theorem}
\normalsize
\begin{proof}
	We show the statement by induction on $d$.
	For $d=1$ we integrate by parts $r$ times integral $I[F,(0,t)]$ and thus we obtain
	\begin{eqnarray*}
		I[F,(0,t)]&=&\int_0^t F(t,\tau_1)\ee^{\ii\omega n_1\tau_1}\D\tau_1\\
		&=&\sum_{k_1=0}^{n-1}\frac{(-1)^{k_1}}{(\ii\omega)^{1+k_1}}\left(\ee^{\ii\omega n_1}\frac{1}{n_1^{k_1+1}}\partial_{\tau_1}^{k_1}F(t,\tau_1)\rule[-4pt]{0.55pt}{17pt}_{\,\tau_1= t}+\frac{-1}{n_1^{k_1+1}}\partial_{\tau_1}^{k_1}F(t,\tau_1)\rule[-4pt]{0.55pt}{17pt}_{\,\tau_1= 0}\right) \\
		&+&\frac{(-1)^r}{(\ii\omega)^r}\frac{1}{n_1^{r}}\int_{0}^{t}\ee^{\ii\omega\tau_1 n_1}\partial_{\tau_1}^{r}F(t,\tau_1)\D\tau_1\\
		&=&
		\sum_{k_1=0}^{r-1}\frac{(-1)^{k_1}}{(\ii\omega)^{1+k_1}}\sum_{\ell=0}^{1}\ee^{\ii\omega v_{\ell}^1n_1}\sum_{\phi \in \Phi^1_{\ell}}A_{k_1}[\phi](n_1)F^{k_1}[\phi](t)+ E_r^{(1)}(t) \\
		&=& \mathcal{S}_r^{(1)}(t)+E_r^{(1)}(t).
	\end{eqnarray*}

	Let $\bm{\tilde{n}}=(n_1,\dots,n_{d-1}), \ \bm{\tilde{k}}=(k_1,\dots,k_{d-1})\in \mathbb{N}^{d-1}$ and $\bm{n}=(n_1,\dots,n_{d}), \ \bm{k}=(k_1,\dots,k_d)\in \mathbb{N}^d$.
	Suppose now the formula (\ref{series+error}) is true for $I[F,\sigma_{d-1}(t)]=\int_{\sigma_{d-1}(t)}F(\bm{\tau})\ee^{\ii\omega \bm{\tilde{n}}^T\bm{\tau}}\D \bm{\tau}$. Then
	\footnotesize
	\begin{eqnarray*}
		&&\int_{S_{d}(t)}F(t,\bm{\tau})\ee^{\ii\omega \bm{n}^T\bm{\tau} }\D \bm{\tau} \stackrel{(1)}{=}\\
		&&\int_{0}^{t}\ee^{(t-\tau_d)\mathcal{L}}\alpha_{n_d}I[F,\sigma_{d-1}(\tau_{d})]\ee^{\ii\omega\tau_{d}n_{d}}\D\tau_{d} \stackrel{(2)}{=} \int_{0}^{t}\ee^{(t-\tau_d)\mathcal{L}}\alpha_{n_d}\left(\mathcal{S}_r^{(d-1)}(\tau_{d})+E_r^{(d-1)}(\tau_{d})\right)\ee^{\ii\omega\tau_{d}n_{d}}\D\tau_{d}\stackrel{(3)}{=}\\
		&&\sum_{|\bm{\tilde{k}}|=0}^{r-d+1}\frac{(-1)^{|\bm{\tilde{k}}|}}{(\ii\omega)^{d-1+|\bm{\tilde{k}}|}}\sum_{\ell=0}^{d-1}\sum_{\bm{\phi}\in \Phi^{d-1}_{\ell}}^{}A_{\bm{\tilde{k}}}[\bm{\phi}](\bm{\tilde{n}})\int_{0}^{t}\ee^{(t-\tau_d)\mathcal{L}}\alpha_{n_d}F^{\bm{\tilde{k}}}[\bm{\phi}](\tau_{d})\ee^{\ii\omega \tau_{d} \bm{n}^T\bm{v}_{\ell}^{d}}\D\tau_{d}+\\
		&& \int_{0}^{t}\ee^{(t-\tau_d)\mathcal{L}}\alpha_{n_d}E_r^{(d-1)}(\tau_{d})\ee^{\ii\omega\tau_d n_d}\D\tau_{d}\stackrel{(4)}{=}\\
		&&\sum_{|\bm{\tilde{k}}|=0}^{r-d+1}\frac{(-1)^{|\bm{\tilde{k}}|}}{(\ii\omega)^{d-1+|\bm{\tilde{k}}|}}\sum_{\ell=0}^{d-1}\sum_{\bm{\phi}\in \Phi^{d-1}_{\ell}}^{}A_{\bm{\tilde{k}}}[\bm{\phi}](\bm{\tilde{n}})\Bigg\{\sum_{k_{d}=0}^{r-d-|\bm{\tilde{k}}|}\frac{(-1)^{k_d}}{(\ii\omega)^{k_d+1}}\frac{1}{(\bm{n}^T\bm{v}_{\ell}^d)^{k_d+1}}\cdot \\
		&&\hspace{0.5cm}\cdot\left(\ee^{\ii\omega t\bm{n}^T\bm{v}_{\ell}^d }\partial^{k_d}_{\tau_d}\left[\ee^{(t-\tau_d)\mathcal{L}}\alpha_{n_d}F^{\bm{\tilde{k}}}[\bm{\phi}](\tau_{d})\right]\,\rule[-8pt]{0.75pt}{20pt}_{\,\tau_d= t}-\partial_{\tau_d}^{k_d}\left[\ee^{(t-\tau_d)\mathcal{L}}\alpha_{n_d}F^{\bm{\tilde{k}}}[\bm{\phi}](\tau_{d})\right]\,\rule[-8pt]{0.75pt}{20pt}_{\,\tau_d=0}\right)+\\
		&&\frac{(-1)^{r-d-|\bm{\tilde{k}}|+1}}{(\ii\omega)^{r-d-|\bm{\tilde{k}}|+1}}\frac{1}{(\bm{n}^T\bm{v}_{\ell}^d)^{r-d-|\bm{\tilde{k}}|+1}} \int_{0}^{t}\partial_{\tau_d}^{r-d-|\bm{\tilde{k}}|+1}\left(\ee^{(t-\tau_d)\mathcal{L}}\alpha_{n_d}F^{\bm{\tilde{k}}}[\bm{\phi}](\tau_{d})\right)\ee^{\ii\omega \tau_{d} \bm{n}^T\bm{v}_{\ell}^{d}}\D\tau_{d}\Bigg\}+\\
		&&\int_{0}^{t}\ee^{(t-\tau_d)\mathcal{L}}\alpha_{n_d}E_r^{(d-1)}(\tau_{d})\ee^{\ii\omega\tau_dn_d}\D\tau_{d}\stackrel{(5)}{=}\\
		&&\sum_{|\bm{k}|=0}^{r-d}\frac{(-1)^{|\bm{k}|}}{(\ii\omega)^{d+|\bm{k}|}}\sum_{\ell=0}^{d-1}\sum_{\phi \in \Phi^{d-1}_{\ell}}A_{\bm{\tilde{k}}}[\bm{\phi}](\bm{\tilde{n}})\frac{1}{(\bm{n}^T\bm{v}_{\ell}^d)^{k_d+1}}\Bigg[\\
		&&\ee^{\ii\omega t\bm{n}^T\bm{v}_{\ell}^d}\partial_{\tau_d}^{k_d}\left[\ee^{(t-\tau_d)\mathcal{L}}\alpha_{n_d}F^{\bm{\tilde{k}}}[\bm{\phi}](\tau_{d})\right]\,\rule[-4pt]{0.75pt}{17pt}_{\,\tau_d=1\cdot t}-\ee^{\ii\omega t\bm{n}^T\bm{v}_d^d}\partial_{\tau_d}^{k_d}\left[\ee^{(t-\tau_d)\mathcal{L}}\alpha_{n_d}F^{\bm{\tilde{k}}}[\bm{\phi}](\tau_{d})\right]\,\rule[-4pt]{0.75pt}{17pt}_{\,\tau_d=t\cdot 0}\Bigg]+\\
		&&\frac{(-1)^{r-d+1}}{(\ii\omega)^{r}}\sum_{|\bm{k}|=r-d+1}^{}\sum_{\ell=0}^{d-1}\sum_{\bm{\phi}\in \Phi^{d-1}_{\ell}}^{}A_{\bm{\tilde{k}}}[\bm{\phi}](\bm{\tilde{n}})\frac{1}{(\bm{n}^T\bm{v}_{\ell}^d)^{k_d}}\int_{0}^{t}\partial_{\tau_d}^{k_d}\left(\ee^{(t-\tau_d)\mathcal{L}}\alpha_{n_d}F^{\bm{\tilde{k}}}[\bm{\phi}](\tau_{d})\right)\ee^{\ii\omega \tau_{d} \bm{n}^T\bm{v}_{\ell}^{d}}\D\tau_{d}\\
		&&+\int_{0}^{t}\ee^{(t-\tau_d)\mathcal{L}}\alpha_{n_d}E_r^{(d-1)}(\tau_{d})\ee^{\ii\omega\tau_dn_d}\D\tau_{d}\stackrel{(6)}{=}\\
		&& \sum_{|\bm{k}|=0}^{r-d}\frac{(-1)^{|\bm{k}|}}{(\ii\omega)^{d+|\bm{k}|}}\sum_{\ell=0}^{d-1}\ee^{\ii\omega t \bm{n}^T\bm{v}_{\ell}^d}\sum_{\bm{\phi}\in \Phi^d_{\ell}}^{}A_{\bm{k}}[\bm{\phi}](\bm{n})F^{\bm{k}}[\bm{\phi}](t)+A_{\bm{k}}[(\phi_1,\dots,\phi_{d-1},0)](\bm{n})F^{\bm{k}}[(\phi_1,\dots,\phi_{d-1},0)](t)+E_r^{(d)}(t)=\\
		&& \sum_{|\bm{k}|=0}^{r-d}\frac{(-1)^{|\bm{k}|}}{(\ii\omega)^{d+|\bm{k}|}}\sum_{\ell=0}^{d}\ee^{\ii\omega t \bm{n}^T\bm{v}_{\ell}^d}\sum_{\bm{\phi}\in \Phi^d_{\ell}}^{}A_{\bm{k}}[\bm{\phi}](\bm{n})F^{\bm{k}}[\bm{\phi}](t)+E_r^{(d)}(t)
	\end{eqnarray*}
	\normalsize
	(we assume that $\sum_{k=0}^{-1}a_k=0$). Throughout the above inductive proof we utilize properties of integral $I[F, \sigma_d(t)]$, operator F, and simple summation identities. More precisely,  where necessary,  in the above identities we have used:
	\begin{enumerate}
		\item[(1)] form  (\ref{function_f_operator})  of function $F$ and Fubini's theorem,
		\item[(2)] induction hypothesis,
		\item[(3)] formula for asymptotic expansion of $I[F,\sigma_{d-1}(t)]$ and identity $\bm{\tilde{n}}^T\bm{v}_\ell^{d-1} +n_d= \bm{n}^T(\bm{v}^{d-1}_\ell,1)$,
		\item[(4)]  integration by parts $r-d-|\bm{\tilde{k}}|+1$ times  and the nonresonance condition (\ref{sum_nonorthogonal_condition}),
		\item[(5)] summation identities $\sum_{|\bm{\tilde{k}}|=0}^{r+1}\sum_{k_d=0}^{r-|\bm{\tilde{k}}|}a_{k_1,\dots,k_{d-1},k_d}=\sum_{|\bm{k}|=0}^{r}a_{k_1,\dots,k_d}$ and \\ $\sum_{|\bm{\tilde{k}}|=0}^{r-d+1}a_{k_1,\dots,k_{d-1},r-d-|\bm{\tilde{k}}|+1}=\sum_{|\bm{k}|=r-d+1}^{}a_{k_1,\dots,k_d}$,
		\item[(6)] Definition \ref{sequence_partial_derivative}, according to which  $\partial_{\tau_d}\left[\ee^{(t-\tau_d)\mathcal{L}}\alpha_{n_d}F^{\bm{\tilde{k}}}[\bm{\phi}](\tau_{d})\right]\,\rule[-4pt]{0.75pt}{17pt}_{\,\tau_d=\phi_d t} = F^{\bm{k}}[\bm{\phi}](t)$, and Definition \ref{numbers_A}.
	\end{enumerate}
We have proved that the integral $I[F,\sigma_d(t)] \sim \mathcal{O}(\omega^{-d})$ can indeed be approximated by the sum (\ref{series_S_expansion}), with an error  $\mathcal{O}(\omega^{-r-1})$ as given by the form (\ref{error_expansion}). 
\end{proof}
 Let us note that we cannot consider an infinite expansion in (\ref{series+error}) because $\mathcal{S}_r^{(d)}(t)$ may not converge to $I[F,\sigma_d(t)]$ as $r\rightarrow \infty$, even if $\omega \gg 1$.

A similar result was first obtained  in \cite{IN_2006}, where the authors provide the asymptotic expansion of a multivariate highly oscillatory integral over a regular simplex. 
However, our result differs in that the integrand is a linear operator, instead of a real-valued function, which makes the formulas for the asymptotic expansion of highly oscillatory integrals more complicated. Furthermore, in the later part of our considerations, the coefficients $A_{\bm{k}}[\bm{\phi}](\bm{n})$   and the error of the expansion must be explicitly derived. 

\section{Error analysis. First four terms of the Modulated Fourier expansion}\label{sec5}
In this section, we perform an  analysis of the error associated with approximating   the solution of equation (\ref{eq:1.5}) using the sum (\ref{partial_expansion}). Additionally, we provide ready-to-use formulas for the first four terms of the partial sum (\ref{partial_expansion}) of the asymptotic expansion.

By  Theorem {\ref{theorem_asymptotic_non_resonances}}, each integral $I[F_{\bm{n}},\sigma_d(t)] $, for vector $\bm{n}$ satisfying the nonresonance condition (\ref{sum_nonorthogonal_condition}), and for sufficiently smooth functions $\alpha_n, \ n=1,\dots, N$, can be expressed as a partial sum of the asymptotic expansion
$$I[F_{\bm{n}},\sigma_d(t)] =S_{r,\bm{n}}^{(d)}(t) + E_{r,\bm{n}}^{(d)}(t) ,$$
where $\mathcal{S}^{(d)}_{r,\bm{n}}(t) $ is the sum corresponding to operator $F_{\bm{n}}$ and
$S_{r,\bm{n}}^{(d)}(t) \sim \mathcal{O}(\omega^{-d})$.
Moreover, for the $r$-th partial sum of the Neumann series we have
$$	u^{[r]}(t)= \sum_{d=0}^{r}T^d\ee^{t\mathcal{L}}u_0=\ee^{t\mathcal{L}}u_0+\sum_{d=1}^{r}\sum_{\bm{n}\in \{1,\dots,N\}^d}I[F_{\bm{n}},\sigma_d(t)]u_0 .$$
Therefore, it is possible to approximate the solution $u(t)$ of equation (\ref{eq:1.5}) with a sum of type (\ref{partial_expansion}). Furthermore, since the Neumann series converges for any given time $t$, the asymptotic expansion is well-defined without requiring time steps.

We show that the sum consists of terms of $\mathcal{S}_{r,\bm{n}}^{(d)}(t) u_0$ which are defined in (\ref{series_S_expansion}) 
$$U_{as}^{[r]}(t) = \ee^{t\mathcal{L}}u_0+\sum_{d=1}^{r}\sum_{\bm{n}\in \{1,\dots,N\}^{d}}\mathcal{S}_{r,\bm{n}}^{(d)}(t) u_0.$$
approximates the solution  of (\ref{eq:1.5}) with error  $\mathcal{O}(\omega^{-r-1})$. We will  need the following lemmas
\begin{lemma}\label{lemma_bounded_kp}
Suppose that Assumption \ref{assumption1} is satisfied.	Let $k\geq 1$ and let  $D(\mathcal{L}^{k}):= H^{2pk}(\Omega)\cap H_0^p(\Omega)$ be the space with the Sobolev norm $\| \ \ \|_{H^{2pk}(\Omega)} $.  Then semigroup $\{\ee^{t\mathcal{L}}\}_{t\geq 0}$ is bounded in the space $D(\mathcal{L}^{k})$.
\end{lemma}
\begin{proof}
We denote $\| \ \ \|_{H^n(\Omega)} =: \| \ \ \|_{n} $.
	To show the boundedness of the semigroup $\{\ee^{t\mathcal{L}}\}_{t\geq 0}$,  with respect to norm $\| \ \ \|_{2pk}$,
	we use $k$ times the following estimate  (see \cite{agmon}, Theorem 9.8)
	\begin{equation*}\label{apriori2}
	\|u\|_{2p+k}\leq C \left(\|\mathcal{L}u\|_{k}+\|u\|_{L^2(\Omega)}\right),
	\end{equation*}
	for   a strongly elliptic differential operator $\mathcal{L}$ of order $2p$ with smooth and bounded coefficients $a_{\bm{p}}$, 
	function $u \in H^{2p+k}(\Omega)\cap H_0^p(\Omega) $,  and  open,  bounded $\Omega \subset \mathbb{R}^m$ with smooth boundary.
	We have
	\begin{eqnarray*}
		\|\ee^{t\mathcal{L}}\|_{D(\mathcal{L}^{k})\leftarrow D(\mathcal{L}^{k})}=\sup_{\|u\|_{2pk}\leq 1}\|\ee^{t\mathcal{L}}u\|_{2pk}&\leq& C_1 \sup_{\|u\|_{2pk}\leq 1}\left(\|\mathcal{L}\ee^{t\mathcal{L}}u\|_{2p(k-1)}
		+\|\ee^{t\mathcal{L}}u\|_{L^{2}(\Omega)}\right)\\
		&\leq& C_1 \sup_{\|u\|_{2pk}\leq 1}\left(\|\mathcal{L}\ee^{t\mathcal{L}}u\|_{2p(k-1)}
		+C_2\right)\\
		&\leq& C_3 \sup_{\|u\|_{2pk}\leq 1}\left(\|\mathcal{L}^2\ee^{t\mathcal{L}}u\|_{2p(k-2)}
		+\|\mathcal{L}\ee^{t\mathcal{L}}u \|_{L^2(\Omega)}+ C_4\right)\\
		&\leq& C_5 \sup_{\|u\|_{2pk}\leq 1}\left(\|\mathcal{L}^2\ee^{t\mathcal{L}}u\|_{2p(k-2)}
		+ C_6\right)\\
		&\vdots& \\
		&\leq& C_{m-1} \sup_{\|u\|_{2pk}\leq 1}\left(\|\mathcal{L}^k\ee^{t\mathcal{L}}u\|_{L^2(\Omega)}
		+ C_m\right)\\
		&\leq& C_{m+1} \sup_{\|u\|_{2pk}\leq 1}\left(\|\ee^{t\mathcal{L}}\mathcal{L}^ku\|_{L^2(\Omega)}
		+ C_{m+2}\right) \leq C(t^\star).
	\end{eqnarray*}
The constants $C_1,C_2,\dots$ depend only on $p$, $\Omega$, $t^\star$ and $\mathcal{L}$. We obtain that the operator $\ee^{t\mathcal{L}}: D(\mathcal{L}^k) \rightarrow D(\mathcal{L}^k)$ is bounded in Sobolev norm $\| \ \ \|_{H^{2pk}(\Omega)}$ by a constant $C(t^\star)$ independent of $t$.
\end{proof}
\begin{lemma}\label{lemma_estimations}
	Suppose that Assumption 1 is satisfied. Further, suppose that $u_0, \ \alpha_n \in H^{2p(r+1)}(\Omega),$ for $ n=1,\dots, N$.   Let   	$I[F_{\bm{n}},\sigma_d(t)]$  be defined by (\ref{series+error}), and let $\bm{n}=(n_1,\dots,n_d)$ be the vector satisfying the nonresonance condition  (\ref{sum_nonorthogonal_condition}). Let  $\mathcal{S}^{(d)}_{r,\bm{n}}(t)$  be the asymptotic expansion (\ref{series_S_expansion})  of the integral  	$I[F,\sigma_d(t)]$, and let    $E^{(d)}_{r,\bm{n}}(t)$ be the error (\ref{error_expansion}) of this expansion. We denote $\| \ \ \|_{H^n(\Omega)} =: \| \ \ \|_{n} $ and $\| \ \ \|_{L^2(\Omega)}=:\| \ \ \|_{L^2}$ . Then there exists a constant $C:=C(\mathcal{L},t^\star, \alpha,u_0)$, independent of $\omega$ and $t$, such that
	\begin{equation} \label{oszacowanie_S}
	\|\mathcal{S}^{(d)}_{r,\bm{n}}(t)u_0\|_{L^2} \leq C \omega^{-d}
	\end{equation}
	\begin{equation}\label{oszacowanie_E}
	 \|E^{(d)}_{r,\bm{n}}(t)u_0\|_{L^2} \leq C \omega^{-r-1}.
	\end{equation}
\end{lemma}
\begin{proof}
	We prove (\ref{oszacowanie_S}). The proof of  (\ref{oszacowanie_E}) is very similar. The expression $F_{\bm{n}}^{\bm{k}}[\bm{\phi}](t)$  from the formula (\ref{series_S_expansion}) can be written explicitly as 
	\begin{eqnarray*}\label{alfa_L_sequence}
		F_{\bm{n}}^{\bm{k}}[\bm{\phi}](t) = \sum_{\bm{m}\leq \bm{k}}^{}{\bm{k}\choose \bm{m}}(-1)^{|\bm{m}|+|\bm{k}|}\mathcal{L}^{r_{d+1}} \alpha_{n_d}\mathcal{L}^{r_{d}}\alpha_{n_{d-1}}\dots \mathcal{L}^{r_{\ell+1}}\alpha_{n_{\ell+1}}\ee^{t\mathcal{L}}\mathcal{L}^{r_{\ell}}\alpha_{n_{\ell}}\dots \mathcal{L}^{r_2}\alpha_{n_1}\mathcal{L}^{r_1},
	\end{eqnarray*}
	for certain  vector $\bm{\phi}\in \Phi^d_\ell$, $\ell=0,1,\dots, d$, and multi-indexes $\bm{k}=(k_1,\dots.k_d), \ \bm{m}=(m_1,\dots,m_d)$, where  ${\bm{k}\choose \bm{m}} = {k_1\choose m_1}\dots {k_d\choose m_d}$.
	No-negative numbers $r_1,\dots r_{d+1}$ satisfy $r_1+\dots+r_{d+1}=|\bm{k}|\leq r$ and $\bm{m}\leq \bm{k}$ means $m_i\leq k_i$ for $i=1,2,\dots ,d$.
Using   inequalities (\ref{Sobolev_estimation}) and $\|L^rv\|_{L^2} \leq C \|v\|_{2pr}$  repeatedly, and Lemma \ref{lemma_bounded_kp},  we obtain
\begin{eqnarray*}
&&\|\mathcal{L}^{r_{d+1}} \alpha_{n_d}\mathcal{L}^{r_{d}}\alpha_{n_{d-1}}\dots \mathcal{L}^{r_{\ell+1}}\alpha_{n_{\ell+1}}\ee^{t\mathcal{L}}\mathcal{L}^{r_{\ell}}\alpha_{n_{\ell}}\dots \mathcal{L}^{r_2}\alpha_{n_1}\mathcal{L}^{r_1}u_0\|_{L^2(\Omega)} \\
&&\leq C_1\|\alpha_{n_d}\|_{2pr_{d+1}}\|\alpha_{n_{d-1}}\|_{2p(r_{d+1}+r_{d})}\dots \|\alpha_{n_1}\|_{2p(|\bm{k}|-r_1)}\|u_0\|_{2p|\bm{k}	|} < C_2.
\end{eqnarray*}
Therefore we have 
$$\|F_{\bm{n}}^{\bm{k}}[\bm{\phi}](t)\|_{L^2} \leq C_2\sum_{\bm{m}\leq \bm{k}}^{}{\bm{k}\choose \bm{m}} = C_2 2^{|\bm{k}|}.$$
We have $A_{\bm{k}^{}}[\bm{\phi}](\bm{n}) \leq 1$ and $|\ee^{\ii\omega t \bm{n}^T\bm{v}_{\ell}^d}|=1$. We obtain the estimate
\begin{eqnarray*}
\|\mathcal{S}^{(d)}_r(t)\|_{L^2}& =& \left\|\sum_{|\bm{k}|=0}^{r-d}\frac{(-1)^{|\bm{k}|}}{(\ii\omega)^{d+|\bm{k}|}}\sum_{\ell=0}^{d}\ee^{\ii\omega t \bm{n}^T\bm{v}_{\ell}^d}\sum_{\bm{\phi}\in \Phi^d_{\ell}}^{}A_{\bm{k}}[\bm{\phi}](\bm{n})F^{\bm{k}}[\bm{\phi}](t)\right\|_{L^2}\\
&\leq&  \sum_{|\bm{k}|=0}^{r-d}\frac{1}{\omega^{d+|\bm{k}|}} C_2 2^{|\bm{k}|} \sum_{\ell=0}^{d} \sum_{\bm{\phi}\in \Phi^d_{\ell}}^{} 1  \leq  \sum_{|\bm{k}|=0}^{r-d}\frac{1}{\omega^{d+|\bm{k}|}} C_2 2^{|\bm{k}|}2^d \leq C \omega^{-d},
\end{eqnarray*}
which completes the proof.
\end{proof}
\begin{theorem}\label{theorem_estimation}
Suppose that Assumption 1 is satisfied. In addition, suppose that $u_0, \ \alpha_n \in H^{2p(r+1)}(\Omega), \ n=1,\dots, N$.
	Let $u(t)$ be the solution of (\ref{eq:1.5}) and $u^{[r]}(t) = \sum_{d=0}^{r}T^d\ee^{t\mathcal{L}}u_0 $ be the $r$-th partial sum of the Neumann series.  We denote  $ \| \ \ \|_{L^2}:=\| \ \ \|_{L^2(\Omega)}$. By $\bm{T}$ we mean the operator $\bm{T} = \sum_{d=0}^{\infty}T^{d}$. Then for $t>0$ the following error estimations hold
	\begin{eqnarray*}
		&(1)&\|u(t)-u^{[r]}(t)\| _{L^2} \leq \sup_{\|v(t)\|_{L^2}\leq 1}\|\mathbf{T}v(t)\|_{L^2}\|T^{r+1}\ee^{t\mathcal{L}}u_0\|_{L^2} = \mathcal{O}(\omega^{-r-1}), \\
		&(2)&\|u(t)-U_{as}^{[r]}(t)\|_{L^2} \leq \sup_{\|v(t)\|_{L^2}\leq 1}\|\mathbf{T}v(t)\|_{L^2}\|T^{r+1}\ee^{t\mathcal{L}}u_0\|_{L^2}+\sum_{d=1}^{r}\sum_{\bm{n}\in \{1,\dots,N\}^d}^{}\|E^{(d)}_{r,\bm{n}}(t)u_0\| _{L^2}  = \mathcal{O}(\omega^{-r-1}).
	\end{eqnarray*}
\end{theorem} 
\begin{proof}
	Let $t>0$.  According to Theorem \ref{theorem_Neumann_Series}, solution $u(t)$ of (\ref{eq:1.5}) can be expressed as the Neumann series $u(t) =  \sum_{d=0}^{\infty}T^d\ee^{t\mathcal{L}}u_0$. Therefore we have
	\begin{eqnarray*}
		\|u(t)-u^{[r]}(t)\| _{L^2} = \left\|\sum_{d=r+1}^{\infty}T^d\ee^{t\mathcal{L}}u_0\right\|_{L^2} = \left\|\sum_{d=0}^{\infty}T^{d}T^{r+1}\ee^{t\mathcal{L}}u_0\right\|_{L^2}  \leq \sup_{\|v(t)\|_{L^2}\leq 1}\|\mathbf{T} v(t)\|_{L^2} \|T^{r+1}\ee^{t\mathcal{L}}u_0\|_{L^2}. 
	\end{eqnarray*}
	By Theorem \ref{theorem_asymptotic_non_resonances} and Lemma \ref{lemma_estimations}, we have $\|T^{r+1}\ee^{t\mathcal{L}}u_0\|_{L^2}\leq\sum_{\bm{n}\in\{1,\dots,N\}^d}\left(\|\mathcal{S}_{r+1,\bm{n}}^{(r+1)}(t)u_0\|_{L^2}+\|E_{r+1,\bm{n}}^{(r+1)}(t)u_0\|_{L^2}\right) $ $= \mathcal{O}(\omega^{-r-1})$.  
	To show (2) we use the triangle inequality
	\begin{eqnarray*}
		\|u(t)-U_{as}^{[r]}(t)\|_{L^2} &\leq& \|u(t)-u^{[r]}(t)\|_{L^2} +\|u^{[r]}(t)-U_{as}^{[r]}(t)\|_{L^2}\\
		& \leq&\sup_{\|v(t)\|_{L^2}\leq 1}\|\mathbf{T}v(t)\|_{L^2} \|T^{r+1}\ee^{t\mathcal{L}}u_0\|_{L^2} +\sum_{d=1}^{r}\sum_{\bm{n}\in \{1,\dots,N\}^d}^{}\|E^{(d)}_{r,\bm{n}}(t)u_0\| _{L^2}, 
	\end{eqnarray*}
	and for each $\bm{n}\in \{1,\dots,N\}^d$,  by Lemma \ref{lemma_estimations}, we have  $\|E^{(d)}_{r,\bm{n}}(t)u_0\|_{L^2} = \mathcal{O}(\omega^{-r-1})$.
 It remains to show that $\sup_{\|v(t)\|_{L^2}\leq 1}\|\mathbf{T}v(t)\|_{L^2}\|$ is bounded. Expression $\bm{T}v(t)$, where $\|v(t)\|_{L^2}\leq 1$ is the solution of the integral equation
\begin{eqnarray}\label{Series_operator}
\psi(t)=v(t)+\int_0^t \ee^{(t-\tau)\mathcal{L}}f(\tau)\psi(\tau)\D\tau.
\end{eqnarray}
 Let   $C_1:= \max_{t \in [0,t^\star]}\|\ee^{t\mathcal{L}}\|_{L^2(\Omega)\leftarrow L^2(\Omega)} $.
By  Gr\"{o}nwall's inequality, solution of (\ref{Series_operator}) can be estimated as follows
\begin{eqnarray}\label{constant_estimation}
\|\mathbf{T}v(t)\|_{L^2}=\|\psi(t)\|_{L^2} \leq \exp(t C_1 \|f\|_\infty ),
\end{eqnarray}
which concludes the proof.
\end{proof}
\begin{remark}
	The upper bound  (\ref{constant_estimation}) of constant $\|\mathbf{T}v(t)\|_{L^2}$ may be large, especially for a  big magnitude of $\|f\|_\infty$. Therefore, let us also consider a different approach. In the proof of inequality (1) of Theorem \ref{theorem_estimation}, we estimate the truncation of the Neumann series 
	\begin{eqnarray}\label{truncation_Neumann_series}
	\|u(t)-u^{[r]}(t)\|_{L^2} =\left\|\sum_{d=0}^{\infty}T^{d}T^{r+1}\ee^{t\mathcal{L}}u_0\right\|_{L^2}. 
	\end{eqnarray}
	Let now $v(t) = T^{r+1}\ee^{t\mathcal{L}}u_0.$ It is  straightforward to verify that $\psi(t)=\sum_{d=0}^\infty T^dv(t)$ is the solution of the following non-homogenious equation 
	\begin{align}\label{equation_remark}
	&\psi'(t)=\mathcal{L}\psi(t)+f(t)\psi(t)+ v'(t)-\mathcal{L}v(t),  \\
	&\psi(0)=v(0) \nonumber.
	\end{align}
	From the form of function $v(t)$ we have
	$$v'(t)-\mathcal{L}v(t) = f(t)T^r\ee^{t\mathcal{L}}u_0 + \mathcal{L}T^{r+1}\ee^{t\mathcal{L}}u_0-\mathcal{L}T^{r+1}\ee^{t\mathcal{L}}u_0 =f(t)T^r\ee^{t\mathcal{L}}u_0 $$
	and $v(0) \equiv 0$. Therefore, equation (\ref{equation_remark}) reads
	\begin{align}\label{equation_remark2}
	&\psi'(t)= (\mathcal{L}+f(t))\psi(t) + f(t)T^r\ee^{t\mathcal{L}}u_0 ,  \\
	&\psi(0)=0 \nonumber,
	\end{align}
	and $f(t)T^r\ee^{t\mathcal{L}}u_0 =\mathcal{O}(\omega^{-r})$.
	Moreover, for a sufficiently small time variable $t$, the solution of (\ref{equation_remark2}) can be written as
	$$\psi(t) = \int_0^t \Phi(t,s)f(s)T^r\ee^{s\mathcal{L}}u_0\D s,$$
	where $\Phi(t,s)$ is the solution of the homogeneous problem 
	$$\Phi'(t,s) = (\mathcal{L}+f(s))\Phi(t,s), \qquad \Phi(s,s) = 1,$$
	and $\Phi(t,s) = \exp(\Omega(t,s))$, where $\Omega(t,s)$ is the Magnus expansion.
	If, for example, (\ref{eq:1.5}) is the Schr\"{o}dinger equation with a time dependent potential, then  
	$\|\Phi(t,s)\|_{L^2}\equiv 1$ and therefore the truncation (\ref{truncation_Neumann_series}) can be estimated in a different manner $$\left\|\sum_{d=0}^{\infty}T^{d}T^{r+1}\ee^{t\mathcal{L}}u_0\right\|_{L^2}=\|\psi(t)\|_{L^2} \leq  t \|f\|_\infty \max_{s\in [0,t]}\|T^r \ee^{s\mathcal{L}}u_0\|_{L^2}.$$ This provides a different estimate of the error constant $\|\mathbf{T}v(t)\|_{L^2}$.
	
\end{remark}

Now we provide a ready-made formula for the sum (\ref{partial_expansion}) for $R=3$, which approximates the solution  $u(t)$ of problem (\ref{eq:1.5}) with error $\mathcal{O}(\omega^{-4})$.
Solution $u(t)$ can be written as
\begin{eqnarray*}
	u(t) = \ee^{t\mathcal{L}}u_0+T^1\ee^{t\mathcal{L}}u_0+T^2\ee^{t\mathcal{L}}u_0+T^3\ee^{t\mathcal{L}}u_0+\bm{T}T^{4}\ee^{t\mathcal{L}}u_0,
\end{eqnarray*}
where $\bm{T} = \sum_{d=0}^{\infty}T^{d}$, and the magnitude of the truncation of the Neumann series $\|\bm{T}T^{4}\ee^{t\mathcal{L}}u_0\|_2 = \mathcal{O}(\omega^{-4})$.

Each term  $T^1\ee^{t\mathcal{L}}u_0$, $T^2\ee^{t\mathcal{L}}u_0$ and $T^3\ee^{t\mathcal{L}}u_0$ we expand    as follows
\begin{eqnarray*}
	T^1\ee^{t\mathcal{L}}u_0 &=&\sum_{n_1\in \{1,\dots,N\}}^{} \left(\mathcal{S}_{3,n_1}^{(1)}(t)+E_{3,n_1}^{(1)}(t)\right)u_0,\\
	T^2\ee^{t\mathcal{L}}u_0&=&\sum_{\bm{n}\in \{1,\dots,N\}^2}^{}\left(\mathcal{S}_{3,\bm{n}}^{(2)}(t)+E_{3,\bm{n}}^{(2)}(t)\right)u_0,\\
	T^3\ee^{t\mathcal{L}}u_0&=&\sum_{\bm{n}\in \{1,\dots,N\}^3}^{}\left(\mathcal{S}_{3,\bm{n}}^{(3)}(t)+E_{3,\bm{n}}^{(3)}(t)\right)u_0.
\end{eqnarray*} 
The partial sum of the asymptotic expansion that approximates  solution $u(t)$ is   $$U_{as}^{[3]}(t)=\ee^{t\mathcal{L}}u_0+\sum_{n_1\in \{1,\dots,N\}}\mathcal{S}_{3,n_1}^{(1)}(t)u_0+\sum_{\bm{n}\in \{1,\dots,N\}^2}\mathcal{S}_{3,\bm{n}}^{(2)}(t)u_0+\sum_{\bm{n}\in \{1,\dots,N\}^3}\mathcal{S}_{3,\bm{n}}^{(3)}(t)u_0,$$  where $\mathcal{S}_{3,n_1}^{(1)}(t)u_0$,  $\mathcal{S}_{3,\bm{n}}^{(2)}(t)u_0$ and $\mathcal{S}_{3,\bm{n}}^{(3)}(t)u_0$, by identity (\ref{series_S_expansion}), are of the following form
\footnotesize
\begin{eqnarray*}
	\mathcal{S}_{3,n_1}^{(1)}(t)u_0 &=& \frac{1}{\ii n_1 \omega }\left(\ee^{\ii n_1\omega t}\alpha_{n_1}\ee^{t\mathcal{L}}-\ee^{t\mathcal{L}}\alpha_{n_1}\right)u_0+\frac{1}{(\ii n_1\omega)^2}\left(\ee^{\ii n_1\omega t }ad_\mathcal{L}^1(\alpha_{n_1})\ee^{t\mathcal{L}}-\ee^{t\mathcal{L}}ad_\mathcal{L}^1(\alpha_{n_1})\right)u_0\\
	&+&\hspace{1cm}\frac{1}{(\ii n_1\omega)^3}\left(\ee^{\ii n_1\omega t }ad_\mathcal{L}^2(\alpha_{n_1})\ee^{t\mathcal{L}}-\ee^{t\mathcal{L}}ad_\mathcal{L}^2(\alpha_{n_1})\right)u_0,
\end{eqnarray*}

\begin{eqnarray*}
	\mathcal{S}_{3,\bm{n}}^{(2)}(t)u_0 &=& \frac{1}{(\ii\omega)^2}\left(\ee^{\ii\omega t(n_1+n_2)}\frac{1}{n_1(n_1+n_2)}\alpha_{n_2}\alpha_{n_1}\ee^{t\mathcal{L}}-\ee^{\ii\omega tn_2}\frac{1}{n_1n_2}\alpha_{n_2}\ee^{t\mathcal{L}}\alpha_{n_1}+\frac{1}{n_2(n_1+n_2)}\ee^{t\mathcal{L}}\alpha_{n_2}\alpha_{n_1}\right)u_0\\
	&+&\frac{1}{(\ii\omega)^3}\Bigg(\ee^{\ii\omega t(n_1+n_2)}\left(\frac{1}{n_1^2(n_1+n_2)}\alpha_{n_2}ad_\mathcal{L}^1(\alpha_{n_1})\ee^{t\mathcal{L}}+\frac{1}{n_1(n_1+n_2)^2}ad_\mathcal{L}^1(\alpha_{n_2}\alpha_{n_1})\ee^{t\mathcal{L}}\right)\\
	&+&\hspace{1cm}\ee^{\ii\omega t n_2}\left(\frac{-1}{n_1^2n_2}\alpha_{n_2}\ee^{t\mathcal{L}}ad_\mathcal{L}^1(\alpha_{n_1})-\frac{1}{n_1n_2^2}ad_\mathcal{L}^1(\alpha_{n_2})\ee^{t\mathcal{L}}\alpha_{n_1}\right)+ \frac{1}{n_1n_2(n_1+n_2)}\ee^{t\mathcal{L}}\alpha_{n_2}ad_\mathcal{L}^1(\alpha_{n_1})\\
	&+&\hspace{1cm}\frac{-1}{n_1(n_1+n_2)^2}\ee^{t\mathcal{L}}ad_\mathcal{L}^1(\alpha_{n_2}\alpha_{n_1})+\frac{1}{n_1n_2^2}\ee^{t\mathcal{L}}ad_\mathcal{L}^1(\alpha_{n_2})\alpha_{n_1}\Bigg)u_0,
\end{eqnarray*}
\begin{eqnarray*}
	S_{3,\bm{n}}^{(3)}(t)u_0&=&\frac{1}{(\ii\omega)^3}\Bigg(\frac{\ee^{\ii\omega t(n_1+n_2+n_3)}}{n_1(n_1+n_2)(n_1+n_2+n_3)}\alpha_{n_3}\alpha_{n_2}\alpha_{n_1}\ee^{t\mathcal{L}}-\frac{\ee^{\ii\omega t(n_2+n_3)}}{n_1n_2(n_2+n_3)}\alpha_{n_3}\alpha_{n_2}\ee^{t\mathcal{L}}\alpha_{n_1}\\
	&+&\hspace{1cm}\frac{\ee^{\ii\omega t n_3}}{n_2n_3(n_1+n_2)}\alpha_{n_3}\ee^{t\mathcal{L}}\alpha_{n_2}\alpha_{n_1}- \frac{1}{(n_1+n_2+n_3)(n_2+n_3)n_3}\ee^{t\mathcal{L}}\alpha_{n_3}\alpha_{n_2}\alpha_{n_1}\Bigg)u_0.
\end{eqnarray*}
\normalsize
Expressions of type $\ee^{t\mathcal{L}}u_0$,  $\ee^{t\mathcal{L}}\alpha_{n_j}u_0$, $\ee^{t\mathcal{L}}ad_\mathcal{L}^1(\alpha_{n_j})u_0$, etc.  can be computed either explicitly, or very efficiently and accurately by using the spectral methods \cite{trefethen} and/or the splitting methods \cite{Splitting_acta}.

To summarize, each terms of the Neumann series $T^d\ee^{t\mathcal{L}}u_0$  can be written as the following sum together with the error
$$T^d\ee^{t\mathcal{L}}u_0= \frac{1}{\omega^d}\sum_{s=0}^{dN}\ee^{\ii\omega s}P_{d,s}^d+  \frac{1}{\omega^{d+1}}\sum_{s=0}^{dN}\ee^{\ii\omega s}P_{d+1,s}^d+\dots+ \frac{1}{\omega^{r}}\sum_{s=0}^{dN}\ee^{\ii\omega s}P_{r,s}^d+ \frac{1}{\omega^{r+1}}\bm{E}_{r}^d .$$
$P_{j,s}^d$  consists of terms  of $\mathcal{S}_{r,\bm{n}}^d(t)u_0$, independent of $\omega$, which is with the frequency $\ee^{\ii\omega s}$ and the magnitude $\frac{1}{\omega^{j}}$. $\bm{E}_{r}^d$  is the whole error associated with the approximation of $T^d\ee^{t\mathcal{L}}u_0$ by the partial sum of the asymptotic series.

It is  convenient to present the idea of the asymptotic expansion in the following table
\footnotesize
\arraycolsep=4pt\def\arraystretch{2}
$$  
\begin{array}{ccccccc}
T^1\ee^{t\mathcal{L}}u_0=  & \frac{1}{\omega}\sum_{s=0}^{N}\ee^{\ii\omega s}P_{1,s}^1& +\frac{1}{\omega^2}\sum_{s=0}^{N}\ee^{\ii\omega s}P_{2,s}^1& +\frac{1}{\omega^3}\sum_{s=0}^{N}\ee^{\ii\omega s}P_{3,s}^1+ & \ldots & +\frac{1}{\omega^r}\sum_{s=0}^{N}\ee^{\ii\omega s}P_{r,s}^1&+\frac{1}{\omega^{r+1}}\bm{E}^{1}_{r}\\
T^2\ee^{t\mathcal{L}}u_0=&   & \quad \frac{1}{\omega^2}\sum_{s=0}^{2N}\ee^{\ii\omega s}P_{2,s}^2&   +\frac{1}{\omega^3}\sum_{s=0}^{2N}\ee^{\ii\omega s}P_{3,s}^2+& \ldots & +\frac{1}{\omega^r}\sum_{s=0}^{2N}\ee^{\ii\omega s}P_{r,s}^2&+\frac{1}{\omega^{r+1}}\bm{E}_r^2 \\ 
T^3\ee^{t\mathcal{L}}u_0= &  & & \quad \frac{1}{\omega^3}\sum_{s=0}^{3N}\ee^{\ii\omega s}P_{3,s}^3+& \ldots& +\frac{1}{\omega^r}\sum_{s=0}^{3N}\ee^{\ii\omega s}P_{r,s}^3&+\frac{1}{\omega^{r+1}}\bm{E}^{3}_{r}  \\
\vdots &   &  &   & \ddots& \vdots \\
T^{r}\ee^{t\mathcal{L}}u_0=& &  &   &  & \quad  \frac{1}{\omega^r}\sum_{s=0}^{rN}\ee^{\ii\omega s}P_{r,s}^r&+\frac{1}{\omega^{r+1}}\bm{E}^{r}_{r} \\
\end{array}
$$
\normalsize
The error of the method consists of truncating the Neumann series and truncating the asymptotic expansion of each term $T^d\ee^{t\mathcal{L}}u_0, \ d=1,\dots r$.

\section{Asymptotic expansion -- integrals with resonance points}\label{sec6}
Theorem \ref{theorem_asymptotic_non_resonances} allows for the asymptotic expansion of  integral $I[F_{\bm{n}},\sigma_d(t)]$    provided that vector $\bm{n}$  satisfies the nonresonance condition (\ref{sum_nonorthogonal_condition}).
However, given the potential function of a more general  form
(\ref{ideal_function_f2}) with negative frequencies,
the solution of  the problem (\ref{eq:1.5}) is 
$$u(t) = \ee^{t\mathcal{L}}u_0+\sum_{d=1}^{\infty}\sum_{\bm{n}\in  \bm{N}^d} I[F_{\bm{n}},\sigma_d(t)]u_0,$$
where $\bm{N}^d$ is a set of all possible vectors $\bm{n}$ 
\begin{eqnarray}\label{set_N_resonance}
\bm{N}^d:= \{-N,\dots,-2,-1,1,2,\dots,N\}^d.
\end{eqnarray}
Set $\bm{N}^d$ comprises vectors $\bm{n}$ which are orthogonal to the simplex $\sigma_d(t)$ or its boundary. Coordinates of such $\bm{n}$  satisfy
$$
n_j + n_{j+1}+\dots +n_{r-1}+n_r = 0,
$$
for certain $1\leq j < r\leq d$. For those  vectors we cannot apply Theorem \ref{theorem_asymptotic_non_resonances}  since  otherwise we would divide by zero in  expressions (\ref{expressions_A}).
Our advantage is, however, that instead of one integral with resonance points, we are  to deal with a sum of such integrals.
In Theorem \ref{theorem_asymptotic_non_resonances}, an increase in the dimension of integral results in an increase in the rate of decay -- while $I[F, \sigma_{d-1}(t)] \sim \mathcal{O}(\omega^{-d+1}), $ $I[F, \sigma_{d}(t)] \sim \mathcal{O}(\omega^{-d})$. This is because at any step of the proof we could integrate by parts each integral since, due to non-resonance assumption, the argument of frequency exponent $\bm{n}^T\bm{v}_\ell^d \neq 0$ and then $\ee^{\ii\omega t \bm{n}^T\bm{v}_\ell^d }$ never disappears. 

At this stage we consider integral $I[F_{\bm{n}},\sigma_d(t)]$ with vector $\bm{n} = (n_1,n_2,n_3,\dots,n_d) \in \bm{N}^d $ whose coordinates satisfy
\begin{eqnarray} \label{sum_resonances}
n_1+\dots +n_d=0 \quad \text{and} \quad n_j+n_{j+1}+\dots +n_{j+r}\neq0 \quad \text{for}   \quad j=1,\dots,d, \ 1\leq j+r \leq d.
\end{eqnarray}
In other words, $\bm{n}$ is  orthogonal only to one the   edge of simplex $\sigma_d(t)$ contained in  line $\{(x_1,x_2,\dots,x_d)\in \mathbb{R}^d : x_1=x_2=\dots=x_d$\}. In such a situation, we have $\bm{n}^T\bm{v}_0^d=n_1+\dots+n_d=0$ and therefore in the proof of Theorem \ref{theorem_asymptotic_non_resonances} we cannot integrate by parts the last, outer integral with $\ee^ {\ii\omega \tau_d\bm{n}^T\bm{v}_0^d}$.
The general case involving the whole set (\ref{set_N_resonance}) is a matter of further research.

Let $\bm{n}_j \in \mathbb{N}^d, \ j=1,\dots, d$ be vectors satisfying (\ref{sum_resonances}), such that
\begin{eqnarray*}
	\bm{n}_1 &=& (n_1,n_2,n_3,\dots,n_{d-1},n_d), \\
	\bm{n}_2 &=& (n_2,n_3,n_4,\dots,n_{d},n_{1}), \\ 
	&\vdots&\\
	\bm{n}_j&=& (n_j,n_{j+1},\dots,n_d,n_1,\dots,n_{j-1}),\\
	&\vdots&\\
	\bm{n}_d &=& (n_{d},n_1,n_2,\dots,n_{d-2},n_{d-1}).
\end{eqnarray*}
If $\bm{n}_1 \in \bm{N}^d$, then vectors $\bm{n}_j, \ j=2,\dots d$  are in set $\bm{N}^d$ (\ref{set_N_resonance}) as well.
Because of assumption (\ref{sum_resonances}), for each $j=1,\dots, d$ we have  $\sum_{j=1}^{d}\bm{n}_j=0.$
From the coordinates of vectors $\bm{n}_j$ we form the following fractional numbers 
\begin{eqnarray*}
	A[\bm{\tilde{n}}_{1}]&=& \frac{1}{n_{1}(n_{1}+n_2)(n_{1}+n_2+n_3)\dots(n_{1}+n_2+n_{3}+\dots+n_{d-1})},	\\
	A[\bm{\tilde{n}}_{2}]&=& \frac{1}{n_{2}(n_{2}+n_3)(n_{2}+n_3+n_4)\dots(n_{2}+n_3+n_{4}+\dots+n_{d})},	\\
	& \vdots&\\
	A[\bm{\tilde{n}}_{k+1}]&=&\frac{1}{n_{k+1}(n_{k+1}+n_{k+2})\dots(n_{k+1}+n_{k+2}+\dots +n_d+n_1+\dots+n_{k-1})}, \\
	& \vdots&\\
	A[\bm{\tilde{n}}_{d}]&=& \frac{1}{n_{d}(n_{d}+n_1)(n_{d}+n_1+n_2)\dots(n_{d}+n_1+n_2\dots+n_{d-3}+n_{d-2})}.
\end{eqnarray*}
(We use  cyclic notation $n_s=n_{s+d}$ for $s\in \mathbb{Z}$).
\begin{lemma}\label{lemma_sum_A}
	Let   $\bm{n}_j$,   $j=1,\dots,d$ be  vectors which satisfy condition  (\ref{sum_resonances}). Then
	$$	\sum_{j=1}^{d}A[\bm{\tilde{n}}_j]=0.$$
\end{lemma}
\begin{proof}
	It is sufficient to apply the partial fraction decomposition to  $A[\bm{\tilde{n}}_{1}]$, by treating  $n_1$ as a variable and other numbers $n_r, \ r\neq 1$ as constants. 
	We aim to express $A[\bm{\tilde{n}}_{1}]$ as
	\begin{eqnarray}\label{decomposition}
	A[\bm{\tilde{n}}_{1}]=\sum_{j=1}^{d-1}\frac{N_j}{n_1+\dots+n_j},
	\end{eqnarray}
	for certain numbers $N_j, \ j=1,\dots d-1$.
	Let us fix index $j$. To determine coefficient $N_j$, we use Heaviside cover-up method.   Substituting $n_1=-(n_2+\dots+n_j)$ we have 
	\footnotesize
	\begin{eqnarray*}
		N_j&=&\frac{1}{n_1(n_1+n_2)\dots(n_1+\dots+n_{j-1})(n_1+\dots+n_{j+1})\dots(n_1+\dots+n_{d-1})}\,\rule[-9pt]{0.75pt}{25pt}_{\,n_1=-(n_2+\dots+n_j)}\\
		&=&\frac{(-1)^{j-1}}{(n_2+\dots+n_j)\dots n_jn_{j+1}(n_{j+1}+n_{j+2})\dots(n_{j+1}+\dots+n_{d-1})}\\
		&=&\frac{1}{(n_{j+1}+\dots+n_{1})\dots(n_{j+1}+n_{j+2}+\dots +n_{j-1})n_{j+1}(n_{j+1}+n_{j+2})\dots(n_{j+1}+\dots n_{d-1})}\\
		&=&\frac{n_{j+1}+\dots+n_d}{n_{j+1}(n_{j+1}+n_{j+2})\dots(n_{j+1}+\dots+n_{d-1})(n_{j+1}+\dots+n_d)(n_{j+1}+\dots+n_1)\dots(n_{j+1}+\dots+n_{j-1})}
	\end{eqnarray*}
	\normalsize
	(in the penultimate equality we used assumption (\ref{sum_resonances})).
	Now since $n_{j+1}+\dots+n_d =-(n_1+\dots+n_j)$, substituting $N_j$   into  (\ref{decomposition}) we obtain
	\begin{eqnarray*}
		A[\bm{\tilde{n}}_{1}]=\sum_{j=1}^{d-1}\frac{-1}{n_{j+1}(n_{j+1}+n_{j+2})\dots(n_{j+1}+\dots+n_{j-1})}= -\sum_{j=1}^{d-1}A[\bm{\tilde{n}}_{j+1}],
	\end{eqnarray*}
	which completes the proof.
\end{proof}

Now let us notice, that if multi-index $\bm{\tilde{k}}=(k_1,\dots,k_{d-1})$ satisfy $|\bm{\tilde{k}}|=0$ and $\bm{\phi}=(1,\dots,1) \in \Phi_0^{d-1}$, then $A_{\bm{\tilde{k}}}[\bm{\phi}](\bm{\tilde{n}}_j) = A[\bm{\tilde{n}}_j]$, where $A_{\bm{\tilde{k}}}[\bm{\phi}](\bm{\tilde{n}}_j)$ are the coefficients from Definition \ref{numbers_A}. In other words, numbers $A[\bm{\tilde{n}}_j], \ j=1,\dots, d$ appear with the first term $\frac{1}{(\ii\omega)^{d-1}}$ of the asymptotic expansion of the integral $I[F_{\bm{\tilde{n}}_j},\sigma_{d-1}(t)]$.

\begin{theorem}\label{theorem_resonances}
	Let $\bm{n}_j \in \bm{N}^d $ be  vector which satisfies (\ref{sum_resonances}) and let  $F_{\bm{n}_j} $ be the operator  defined in (\ref{function_f_operator}) with corresponding  vector $\bm{n}_j, \ j=1,\dots,d$.
	Then 
	\begin{eqnarray*}
		\sum_{j=1}^{d}\int_{\sigma_d(t)}F_{\bm{n}_j}(t,\bm{\tau})\ee^{\ii\omega\bm{n}_j^T \bm{\tau} }\D \bm{\tau} \sim \mathcal{O}(\omega^{-d}).
	\end{eqnarray*}
	In other words,   the sum of these integrals with resonance points over a simplex $\sigma_d(t) $ decays in the same manner as an integral over the same domain without resonance points.
\end{theorem}

\begin{proof}
	Let $F_{\bm{n}_j}(t,\bm{\tau}) = \ee^{(t-\tau_d)\mathcal{L}}\alpha_{n_{d,j}}\ee^{(\tau_d-\tau_{d-1})\mathcal{L}}\alpha_{n_{d-1,j}}\dots \ee^{(\tau_{2}-\tau_1)\mathcal{L}}\alpha_{n_{1,j}}\ee^{\tau_1\mathcal{L}}$ be the operator with corresponding  vector $\bm{n}_j=(n_{	1,j},n_{2,j},\dots, n_{d,j})$ , and let $\mathcal{S}^{(d-1)}_{r,\bm{\tilde{n}}_j},  E_{r,\bm{\tilde{n}}_j}^{(d-1)}$ be the partial sum of the asymptotic series and the error of approximation of integral $I[F_{\bm{\tilde{n}}_j},\sigma_{d-1}(t)]$. We use Fubini's theorem and then we apply Theorem \ref{theorem_asymptotic_non_resonances} to expand asymptotically $I[F_{\tilde{\bm{n}}_j},\sigma_{d-1}(t)].$ It is possible since the nonresonance condition is violated only for the vectors $\bm{n}_j \in \mathbb{N}^d, \ j=1,\dots, d$. 
	\begin{eqnarray*}
		&&\sum_{j=1}^{d}\int_{\sigma_{d}(t)}F_{\bm{n}_j}(t,\bm{\tau})\ee^{\ii\omega \bm{n}_j^T\bm{\tau} }\D \bm{\tau} =\\
		&&\sum_{j=1}^{d}\int_{0}^{t}\ee^{(t-\tau_d)\mathcal{L}}\alpha_{n_{d,j
		}}I[F_{\bm{\tilde{n}}_j},\sigma_{d-1}(\tau_{d})]\ee^{\ii\omega\tau_{d}n_{d,j}}\D\tau_{d} =\\
		&&\sum_{j=1}^{d} \int_{0}^{t}\ee^{(t-\tau_d)\mathcal{L}}\alpha_{n_{d,j}}\left(\mathcal{S}_{r,\bm{\tilde{n}}_j}^{(d-1)}(\tau_{d})+E_{r,\bm{\tilde{n}}_j}^{(d-1)}(\tau_{d})\right)\ee^{\ii\omega\tau_{d}n_{d,j}}\D\tau_{d}\stackrel{}{=}\\
		&&\underbrace{\sum_{j=1}^{d}\sum_{|\bm{\tilde{k}}|=0}^{r-d+1}\frac{(-1)^{|\bm{\tilde{k}}|}}{(\ii\omega)^{d-1+|\bm{\tilde{k}}|}}\sum_{\ell=1}^{d-1}\sum_{\bm{\phi}\in \Phi^{d-1}_{\ell}}^{}A_{\bm{\tilde{k}}}[\bm{\phi}](\bm{\tilde{n}}_j)\int_{0}^{t}\ee^{(t-\tau_d)\mathcal{L}}\alpha_{n_{d,j}}F_{\bm{\tilde{n}}_j}^{\bm{\tilde{k}}}[\bm{\phi}](\tau_{d})\ee^{\ii\omega \tau_{d} \bm{n}_j^T\bm{v}_{\ell}^{d}}\D\tau_{d}}_{=:P_1(t)}+\\
		&&\underbrace{\sum_{j=1}^{d}\sum_{|\bm{\tilde{k}}|=0}^{r-d+1}\frac{(-1)^{|\bm{\tilde{k}}|}}{(\ii\omega)^{d-1+|\bm{\tilde{k}}|}}A_{\bm{\tilde{k}}}[(1,\dots,1)](\bm{\tilde{n}}_j)\int_{0}^{t}\ee^{(t-\tau_d)\mathcal{L}}\alpha_{n_{d,j}}\bm{F}_{\bm{\tilde{n}}_j}^{\bm{\tilde{k}}}[(1,\dots,1)](\tau_{d})\ee^{\ii\omega \tau_{d} \bm{n}_j^T\bm{v}_{0}^{d}}\D\tau_{d}}_{=:P_2(t)}+\\
		&& \sum_{j=1}^{d}\int_{0}^{t}\ee^{(t-\tau_d)\mathcal{L}}\alpha_{n_{d,j}}E_{r,j}^{(d-1)}(\tau_{d})\ee^{\ii\omega\tau_d n_{d,j}}\D\tau_{d}.\\
	\end{eqnarray*}
	Now since vectors $\bm{n}_j, \ j=1,\dots, d$ satisfy (\ref{sum_resonances}), we have $\bm{n}_j^T\bm{v}_0^d=0$, so $\ee^{\ii\omega\tau_d\bm{n}_j^T\bm{v}_0^d}=1$ and therefore we cannot expand asymptotically expression $P_2(t)$. However, if  $|\bm{\tilde{k}}|=0$ then for each $j$ holds  $\ee^{(t-\tau_d)\mathcal{L}}\alpha_{n_{d,j}}F_{\bm{\tilde{n}_j}}^{\bm{\tilde{k}}}[(1,\dots,1)](\tau_{d}) = \ee^{(t-\tau_d)\mathcal{L}}\alpha_{1}\alpha_{2}\dots\alpha_{d}\ee^{\tau_d\mathcal{L}}$ since functions $\alpha_j$ commute with each other. As a consequence, in expression $P_2$ , by Lemma \ref{lemma_sum_A}, terms with $|\bm{\tilde{k}}|=0$ vanish, so $P_2$ is  equal to
	$$P_2(t)=\sum_{j=0}^{d}\sum_{|\bm{\tilde{k}}|=1}^{n-d+1}\frac{(-1)^{|\bm{\tilde{k}}|}}{(\ii\omega)^{d-1+|\bm{\tilde{k}}|}}A_{\bm{\tilde{k}}}[(1,\dots,1)](\bm{\tilde{n}}_j)\int_{0}^{t}\ee^{(t-\tau_d)\mathcal{L}}\alpha_{n_{d,j}}F_{\bm{\tilde{n}}_j}^{\bm{\tilde{k}}}[(1,\dots,1)](\tau_{d})\D\tau_{d}$$
	and thus $P_2(t)\sim \mathcal{O}(\omega^{-d})$.
	Expression $P_1(t)$ we integrate by parts according to Theorem \ref{theorem_asymptotic_non_resonances}
	since each integrals of expressions $P_1(t)$ has no resonance points, therefore $P_1(t)\sim  \mathcal{O}(\omega^{-d})$ 
	and consequently  $P_1+P_2\sim \mathcal{O}(\omega^{-d}).$
\end{proof}

In the asymptotic series of $\sum_{j=0}^{d}I[F_{n_j},\sigma_d(t)]$, in expressions which were denoted by $P_2$  in the proof of Theorem \ref{theorem_resonances}, appear terms with integral
$$\int_{0}^{t}\ee^{(t-\tau_d)\mathcal{L}}\alpha_{n_{d,j}}F_{\bm{\tilde{n}}_j}^{\bm{\tilde{k}}}[(1,\dots,1)](\tau_{d})\D\tau_{d}$$
yet they are not highly oscillatory, so we expect we can approximate them effortlessly and effectively, for example by using Gauss-Legendre quadrature.

\begin{exam}
	Consider set $\bm{N}^2=\{\bm{n}_1=(-1,1),\bm{n}_2=(1,-1)\}$ and two integrals 
	\begin{eqnarray*}
		I[F_{\bm{n}_1},\sigma_2(t)] \quad \text{and} \quad 	I[F_{\bm{n}_2},\sigma_2(t)].
	\end{eqnarray*}
	Then $	I[F_{\bm{n}_1},\sigma_2(t)] +	I[F_{\bm{n}_2},\sigma_2(t)]\sim \mathcal{O}(\omega^{-2}) $  and the first term of the asymptotic expansion of  $	I[F_{\bm{n}_1},\sigma_2(t)] +	I[F_{\bm{n}_2},\sigma_2(t)] $ is equal to
		\begin{eqnarray*}
			&&\frac{1}{(\ii\omega)^2}\Bigg(\alpha_{1}\ee^{t\mathcal{L}}\alpha_{-1}\ee^{\ii\omega t} +\alpha_{-1}\ee^{t\mathcal{L}}\alpha_{1}\ee^{-\ii\omega t}-2\ee^{t\mathcal{L}}\alpha_1\alpha_{-1}+\int_{0}^{t}\ee^{(t-\tau_2)\mathcal{L}}\alpha_{-1}ad^1_\mathcal{L}(\alpha_{1})\ee^{\tau_2\mathcal{L}}\D\tau_2\\
			&&\hspace{1cm}+\int_{0}^{t}\ee^{(t-\tau_2)\mathcal{L}}\alpha_{1}ad_\mathcal{L}^1(\alpha_{-1})\ee^{\tau_2\mathcal{L}}\D\tau_2\Bigg).\\
	\end{eqnarray*}
	\normalsize
	Occurring integrals are not highly oscillatory and can be computed,   for example, by  Gauss-Legendre quadrature with high accuracy. 
\end{exam}
To summarise this chapter, it is much more difficult to provide formulas for coefficients of the asymptotic expansion for integrals with resonance points. However, Theorem \ref{theorem_resonances} describes the asymptotic behaviour of terms from the Neumann series, and it seems possible to use this fact to construct quadrature rules based on the Filon method.
\section{A highly oscillatory wave equation}\label{sec7}
The proposed method can also be successfully applied to the approximation  of highly oscillatory equations with a second time derivative.
 Consider the following second-order PDE

\begin{align} \label{main_equation}
&\partial_{tt}u=\mathcal{L} u(x,t)+f(x,t)u(x,t),\qquad t\in [0,t^\star],\ x\in\Omega\subset\mathbb{R}^m,\\ \nonumber
&u(x,0)=u_1(x), \quad \partial_t u(x,0) = u_2(x), \\ \nonumber
& u=0 \ \text{on} \   \partial\Omega \times [0,t^\star],
\end{align}
with function $f$ given in (\ref{ideal_function_f}). We write (\ref{main_equation}) as a first-order system
\arraycolsep=4pt\def\arraystretch{1.2}
\begin{equation}\label{eq:matrix}
\partial_t\left[  \begin{array}{c} \nonumber
u \\
v
\end{array} \right] = \left[  \begin{array}{cc}
0 & \mathcal{I} \\
\mathcal{L} & 0
\end{array} \right]	\left[  \begin{array}{c}
u \\
v
\end{array} \right]+ f\left[  \begin{array}{cc}
0 & 0 \\
\mathcal{I} & 0
\end{array} \right]\left[  \begin{array}{c}
u \\
v
\end{array} \right],
\end{equation}
where   $v  = \partial_t u$. Thus
\begin{eqnarray*} 
	\partial_t\underbrace{\left[  \begin{array}{c} 
			u \\ 
			v
		\end{array} \right]}_{\varphi} = \underbrace{\left[  \begin{array}{cc}
			0 & \mathcal{I} \\
			\mathcal{L} & 0
		\end{array} \right]	}_{A}\left[  \begin{array}{c}
		u \\
		v
	\end{array} \right]+\sum_{n=1}^{N}\ee^{\ii n\omega t}\underbrace{ \left[  \begin{array}{cc}
			0 & 0 \\
			\alpha_n & 0
		\end{array} \right] }_{\beta_n}\left[  \begin{array}{c}
		u \\
		v
	\end{array} \right],
\end{eqnarray*}
and therefore
\begin{equation}\label{first_order}
\partial_t\varphi = A\varphi+h\varphi, \quad \varphi(x,0)=[u_1(x),u_2(x)], \quad A[u,v]^T = [v,\mathcal{L}u]^T, \quad \beta_n[u,v]^T=[0,\alpha_n u]^T,
\end{equation}
where $h(x,t)= \sum_{n=1}^{N}\ee^{\ii n\omega t}\beta_n(x)$ is a highly oscillatory function  and $\varphi$ is a vector valued function.
Let $\mathcal{L}$ be a second-order differential operator which has symmetric form
$$\mathcal{L}u = \sum_{i,j=0}^{m}\partial_{x_j}\left(a_{ij}\partial_{x_i}u\right)-cu,$$
where $a_{ij}=a_{ji}, \ i,j=1,\dots,m$ and $c\geq 0$. For simplicity, assume that   $\|\alpha_n\|_\infty < \infty, \ \forall n=1,\dots,N$. 
By applying  Duhamel's formula, we write (\ref{first_order}) as

\begin{eqnarray}\label{eq:Duhamel2}
\varphi(t) = \ee^{tA}\varphi_0+\int_{0}^{t}\ee^{(t-\tau)A}h(\tau)\varphi(\tau)\D\tau.
\end{eqnarray}
Operator  $A: D(A):= \left[ H^{2}(\Omega)\cap H_0^1(\Omega) \right]\times H_0^1(\Omega)\rightarrow H_0^1(\Omega)\times L^2(\Omega)$  is the infinitesimal generator of a $C_0$-semigroup $\{\ee^{tA}\}$ on $H_0^1(\Omega)\times L^2(\Omega)$ \cite{EVANS}. Using the same arguments as in the proof of Theorem \ref{theorem_Neumann_Series}, one can show that the Neumann series converges absolutely and uniformly  in the norm of space $H_0^1(\Omega)\times L^2(\Omega)$ to the solution of equation (\ref{eq:Duhamel2}).

\section{Numerical examples}\label{sec8}
In this section, we present the application of the method to  equations of type (\ref{eq:1.5}) and (\ref{main_equation}). For each of the equations, it is possible to find an analytical solution to compare them  accurately with a numerical approximation.
The $L^2$ norm of the error is considered in any presented example.
For each equation, the  solution is approximated by a partial sum of the asymptotic expansion 
\begin{equation}\label{partial_expansion_example}
u(x,t) \approx p_{0,0}(x,t) +  \sum_{r=1}^{R}\frac{1}{\omega^r}\sum_{s=0}^{S}p_{r,s}(x,t)e^{i s \omega t},
\end{equation}
for different $R$ and $\omega$.
\\

\noindent \emph{Example 1}.\\
We first consider the following equation 
\begin{align}\label{heat_example} \nonumber
&\partial_{t}u=(1-x^2)^4 \partial_{xx}^2 u+f(x,t)u(x,t),\qquad t\in [0,3],\ x\in(-1,1),\\ 
&u(x,0)=u_0(x), \\
&u(-1,t) = 0=u(1,t) ,\nonumber
\end{align}
where initial condition $u_0$ and highly oscillatory potential $f$ take the forms
\begin{equation*}
u_0(x)=\left\{
\begin{array}{@{}ll@{}}
\ee^{-\frac{1}{1-x^2}}& \text{if}\ x\in (-1,1), \\
0& \text{otherwise},
\end{array}\right.
\end{equation*}
\small
\begin{eqnarray*}
	f(x,t) &=& \underbrace{\frac{1}{\omega^2}\left((1-x^2)^4+4\ii \omega x(1-x^2)^2+2\omega^2(1-3x^4)\right)}_{\alpha_0(x)}+\ee^{\ii\omega t}\underbrace{\frac{1}{\omega^2}\left(-2(1-x^2)^4-4\ii x\omega (1-x^2)^2+x\omega^2\right)}_{\alpha_1(x)}\\
	&+&\ee^{2\ii\omega t}\underbrace{\frac{1}{\omega^2}(1-x^2)^4}_{\alpha_2(x)}.
\end{eqnarray*}
\normalsize
\begin{figure}
	\begin{center}
		\includegraphics[width=0.65\textwidth]{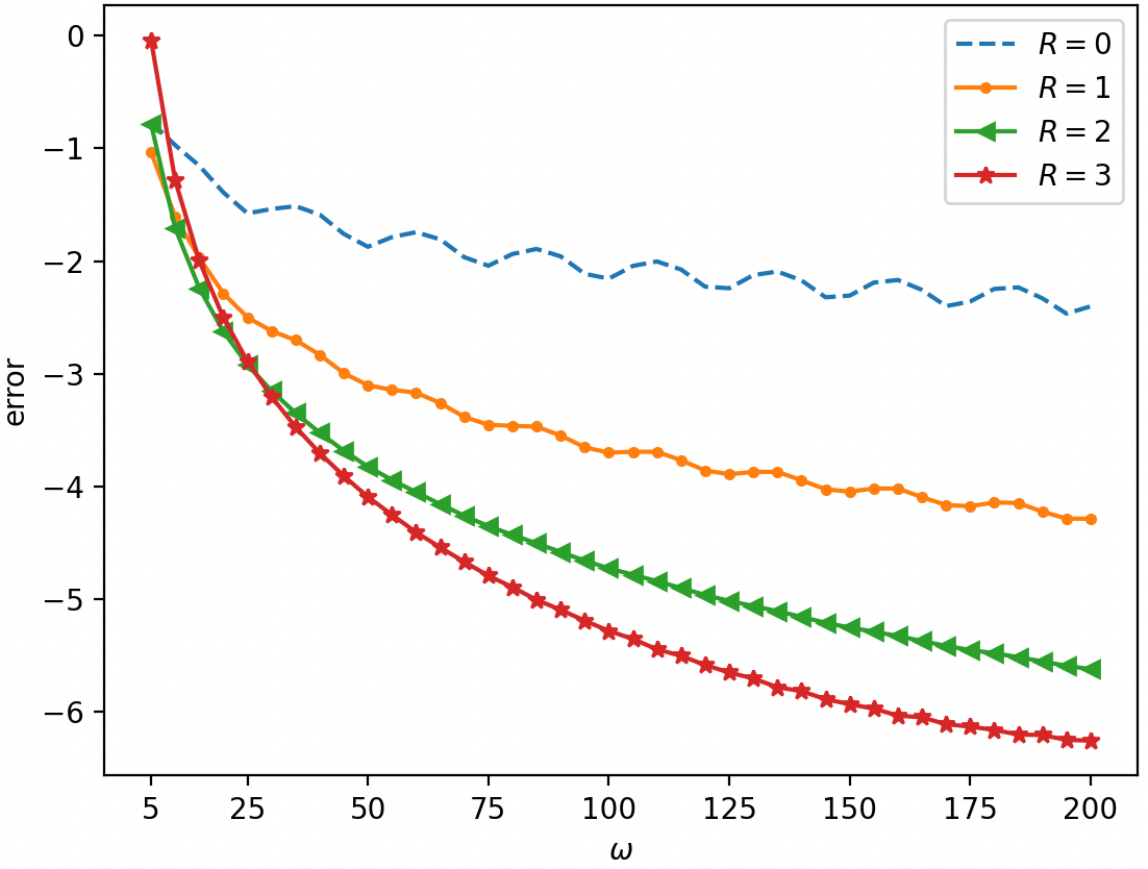}  
		\caption{$L^2$ norm of the error of the method for the equation  (\ref{heat_example}) for $t=3$. We use a base--10 log scale.
		} \label{figure_heat}
	\end{center}
\end{figure}
\noindent The  solution of (\ref{heat_example}) is equal to
$$u(x,t) = \ee^{-\frac{\ii\ee^{\ii\omega t}x}{\omega}+\frac{\ii x}{\omega}}u_0(x)$$
and differential operator $\mathcal{L}$ is of the form 
$$\mathcal{L} = (1-x^2)^4\partial_{xx}^2 + \alpha_0.$$
To approximate the solution we take the first four terms of the Neumann series
$$u^{[3]}(t) = \ee^{t\mathcal{L}}u_0(x)  + T^1\ee^{t\mathcal{L}}u_0+T^2\ee^{t\mathcal{L}}u_0+T^3\ee^{t\mathcal{L}}u_0.$$
Subsequently, we expand  asymptotically  $T^1\ee^{t\mathcal{L}}u_0$, $T^2\ee^{t\mathcal{L}}u_0$ and $T^3\ee^{t\mathcal{L}}u_0$  with error $\mathcal{O}(\omega^{-4}).$
In other words, we approximate the solution of (\ref{heat_example}) by the partial sum of asymptotic expansion (\ref{partial_expansion_example})  with  the first four terms.  
Figure \ref{figure_heat} and Table \ref{table_heat} show the approximation error of the solution $u(x,3)$ for different values of $\omega$ and different  numbers $R$.
\begin{table}[ht]
	\parbox{.50\linewidth}{
		\centering
  \begin{tabular}{|c |c |c|c|}
			\hline
			&$R=0$  & $R=1$ &$R=2$\\ [0.5ex] 
			\hline 
			$\omega=10$ &9.33e-03 & 5.28e-03 &5.55e-05\\ 
			\hline
			$\omega=100$ &5.33e-04 & 4.87e-05& 4.94e-08\\
			\hline
			$\omega=1000$ &4.96e-05 & 7.64e-08&4.86e-11 \\
			\hline
		\end{tabular}\label{table_time_dependent}
		\caption{Error of the method --  equation (\ref{heat_time_dep_example}).}\label{table_heat_time}
	}
	\hfill
	\parbox{.50\linewidth}{
		\centering
		\begin{tabular}{|c |c |c| c|c|}
			\hline
			& $R=0$ & $R=1$ & $R=2$ &$R=3$\\ [0.5ex] 
			\hline 
			$\omega=10$ & 1.12e-01 & 3.17e-02 & 2.75e-02 &7.38e-02\\ 
			\hline
			$\omega=100$ & 1.23e-02 & 3.06e-04 & 2.68e-05 &7.36e-06\\
			\hline
			$\omega=1000$ & 1.27e-03 & 3.15e-06 & 2.68e-08&1.76e-09 \\
			\hline
		\end{tabular} 
		\caption{Error of the method -- equation (\ref{heat_example})} \label{table_heat}
	}	
\end{table}
\\

\noindent \emph{Example 2}.\\
As mentioned, the method can be applied to the potential $f$ with time--dependent functions $\alpha_n(x,t)$. Indeed, consider the following equation
\begin{align}\label{heat_time_dep_example} \nonumber
&\partial_{t}u=\partial^2_{xx} u+\ee^{\ii\omega t}\sin(t)u(x,t),\qquad t\in [0,5],\ x\in (0,2\pi), \\ 
&u(x,0)=\sin(x), \\
&u(0,t) =0=u(2\pi,t) .\nonumber 
\end{align}
The solution of (\ref{heat_time_dep_example}) is equal to
$$u(x,t) = \ee^{-t-1/(\omega^2-1)+\ee^{\ii\omega t}\cos(t)/(\omega^2-1)-\ii\ee^{\ii\omega t}\omega\sin(t)/(\omega^2-1)}\sin(x).$$
Operator $\mathcal{L}= \partial_{xx}^2$ and the potential $f$ is $f(x,t)=\ee^{\ii\omega t}\alpha(x,t)$, where $\alpha(x,t)= \sin(t)$ is time--dependent function. We approximate the  solution by taking the first three terms of the Neumann series 
$$u^{[2]}(t)= \ee^{t\mathcal{L}}u_0+T^1\ee^{t\mathcal{L}}u_0+T^2\ee^{t\mathcal{L}}u_0.$$
To expand asymptotically integrals $T^1\ee^{t\mathcal{L}}u_0$ and $T^2\ee^{t\mathcal{L}}u_0$ we utilize the folowing generalization of Lemma \ref{lemma_ady}. 
\begin{lemma}\label{lemma_ady_time}
	Let $\alpha(\tau) \in C^k([0,t^\star], D(\mathcal{L}^k))$. Then the	$k$-th time derivative of  expression $\ee^{(t-\tau)\mathcal{L}}\alpha(\tau)\ee^{\tau\mathcal{L}}$ is 
	\begin{eqnarray*}\label{equality_ady_time_dependent}
		&&	\partial_\tau^k\left(\ee^{(t-\tau)\mathcal{L}}\alpha(\tau)\ee^{\tau\mathcal{L}}\right)=\sum_{\ell=0}^{k}(-1)^\ell \binom{k}{\ell} \ee^{(t-\tau)\mathcal{L}}ad_\mathcal{L}^\ell\big(\alpha^{(k-\ell)}(\tau)\big)\ee^{\tau\mathcal{L}}.
	\end{eqnarray*}
\end{lemma}
\noindent  The proof can be found in Appendix.
Table \ref{table_heat_time} presents the error of the method up to the second iteration.

\begin{table}[ht]
\parbox{.50\linewidth}{
		\centering
	\begin{tabular}{|c|c|c|c|} 
		\hline
		& $R=0$ &$R=1$  & $R=2$ \\ [0.5ex] 
		\hline 
		$\omega=10$ & 3.17e-02 &3.17e-02& 1.71e-03 \\ 
		\hline
		$\omega=100$ & 5.45e-04 &5.45e-04 & 2.00e-07 \\
		\hline
		$\omega=1000$ & 5.54e-06 & 5.54e-06& 1.98e-11 \\
		\hline
	\end{tabular}
	\caption{Error of the method- equation (\ref{wave_example})}\label{table_wave}
}
	\hfill
 \parbox{.50\linewidth}{
		\centering
	\begin{tabular}{|c |c |c| c|c|}
			\hline
			& $R=0$ & $R=1$ & $R=2$ &$R=3$\\ [0.5ex] 
			\hline 
			$\omega=10$ & 3.58e-00 & 3.47e-01 & 2.22e-02 &1.07e-03\\ 
			\hline
			$\omega=100$ & 1.00e-01 & 2.64e-04 & 4.62e-07 &6.07e-10\\
			\hline
			$\omega=1000$ & 1.80e-02 & 8.41e-06 & 2.62e-09&6.14e-13 \\
			\hline
		\end{tabular} 
	\caption{Error of the method- equation (\ref{biharmonic_example})}\label{table_biharmonic}
}
\end{table}
\vspace{0.4cm}

\noindent \emph{Example 3}.\\
Consider now the wave equation with potential with negative frequencies
\begin{align} \label{wave_example}
&\partial_{tt}u=\partial_{xx} \nonumber u+f(x,t)u(x,t),\qquad t\in [0,1],\ x\in (-L,L), \quad L=10,\\ 
&u(x,0)=\ee^{-x^2(1/2+1/\omega^2)}, \quad \partial_t u(x,0) = 0,\\
&u(-L,t) =u(L,t) ,\nonumber \\
& \partial_{t}u(-L,t) =\partial_{t}u(L,t),   \nonumber
\end{align}
where function $f$ takes the form
$$f(x,t) = \left(1-x^2+\frac{(2+x^2\omega^2-4x^2)\cos(\omega t)}{\omega^2}-\frac{4x^2\cos^2(\omega t)}{\omega^4}+\frac{x^4\sin^2(\omega t)}{\omega^2}\right).$$
The solution of (\ref{wave_example}) is equal to $$u(x,t) = \ee^{-\cos(\omega t)x^2/\omega^2}\ee^{-x^2/2}.$$
Due to the presence of the resonance points, we do not have a general formula for the asymptotic expansion of the integrals that appears in the Neumann series  in this case. Nevertheless, we can employ Theorem \ref{theorem_asymptotic_non_resonances} and Theorem \ref{theorem_resonances} to approximate only terms $T^1\ee^{tA}$ and $T^2\ee^{tA}$, where $A$ is the linear operator of form (\ref{first_order}).
Table \ref{table_wave}  presents the error of the method for a different oscillatory parameter $\omega$ and different number $R$.\\
\\
\noindent \emph{Example 4}.\\
In the last example, we consider the equation with the biharmonic operator 
\begin{eqnarray}\label{biharmonic_example} \nonumber
	&&\partial_tu = \partial_{xxxx}^4 u+\ee^{\ii\omega t} u(x,t), \quad x\in (0, \pi), \quad t\in[0,1],\\
	&& u(x,0) = u_0(x) = \sin(x)\exp\left(-\frac{\ii}{\omega}\right),\\
	&& u(0,t)= u(\pi,t)=0, \nonumber
\end{eqnarray}
with periodic boundary conditions.
The solution is 
\begin{align*}
    u(x,t) = \ee^{t-\ii\ee^{\ii\omega t}/\omega} \sin(x).
\end{align*}
Table \ref{table_biharmonic} presents the error of the method for different parameter $\omega$ and different partial sums of the asymptotic expansion.
\begin{remark}
	One can notice that equations (\ref{heat_example}), (\ref{wave_example}) and (\ref{biharmonic_example}) do not fully satisfy Assumption \ref{assumption1}. Indeed,  linear operator $\mathcal{L}$ form equation (\ref{heat_example}) does not satisfy strong ellipticity condition for $x$ near $-1$ or $1$, and solutions $u(x,t) $ of equations (\ref{wave_example}) and (\ref{biharmonic_example})  does not satisfy the zero boundary conditions.    These are just examples to illustrate the proposed methodology and suggest that the method can  be applied to a wider class of differential equations.   
\end{remark}

\section{Concluding remarks}\label{sec9}
In this paper, we have demonstrated that the solution to a highly oscillatory differential equation (\ref{eq:1.5}) can be approximated by the sum in the form (\ref{partial_expansion}).  We started by showing that the Neumann series converges to the solution of the equation in the Sobolev space $H^{2p}(\Omega)\cap H_0^p(\Omega)$. Then, by using partial integration, we expand asymptotically the integrals appearing in the Neuman series. We show that the asymptotic method can be used for differential equations with a second time derivative. Numerical experiments illustrate the proposed method and confirm the theoretical studies.

There are several plans for extending the proposed approach. 
The primary concern is to establish rigorous estimations for the error formulas that arise in the asymptotic expansion of highly oscillatory integrals. This task is particularly challenging, especially when dealing with a general differential operator $\mathcal{L}$ in the form of (\ref{diff_operator}). It should be noted that the partial sum $\mathcal{S}_r^{(d)}(t)$, used to approximate the integral $I[F_{\bm{n}},\sigma_d(t)]$, may be divergent as $r\rightarrow \infty$. A significant advancement in the approximation of highly oscillatory PDEs would involve determining the maximum value of $r^\star$, beyond which the method's error begins to increase.

To approximate each integral of the form (\ref{I[F,S_d(t)]}), we used the asymptotic method, which represents a simple form of quadrature rules tailored to highly oscillatory integrals. If adding successive terms of the sum (\ref{partial_expansion}) does not improve the expected accuracy, it is possible to modify the proposed method by incorporating more sophisticated formulas, such as Filon-type methods. This approach should further reduce the approximation error.

The next  step would be to provide the asymptotic expansion of the solution for the potential in the form (\ref{ideal_function_f2}). Furthermore, it would be interesting to investigate whether it is possible to represent the solution of the equation as an asymptotic series for an even more general highly oscillatory potential given by
\begin{equation*}
f(x,t) = \sum_{n=1}^{N}\alpha_n(x,t)\ee^{\ii \omega g_n(t)}, \quad \omega \gg 1, \quad N\in \mathbb{N},
\end{equation*}
where $\alpha_n$ and $g_n$ are sufficiently smooth functions. If such an expansion is possible, the method could be applicable to a wide range of highly oscillatory differential equations.

\appendix

\section{}\label{secA1}
\begin{lemma}\label{lemma_ady_time}
	The	$k$-th time derivative of  expression $\ee^{(t-\tau)\mathcal{L}}\alpha(\tau)\ee^{\tau\mathcal{L}}$ is equal to
	\begin{eqnarray}\label{equality_ady_time_dependent}
	&&	\partial_\tau^k\left(\ee^{(t-\tau)\mathcal{L}}\alpha(\tau)\ee^{\tau\mathcal{L}}\right)=\sum_{\ell=0}^{k}(-1)^\ell \binom{k}{\ell} \ee^{(t-\tau)\mathcal{L}}ad_\mathcal{L}^\ell\big(\alpha^{(k-\ell)}(\tau)\big)\ee^{\tau\mathcal{L}}.
	\end{eqnarray}
\end{lemma}
\begin{proof}
	We prove the statement by induction on $k$.
	Let $k=1$. 
	\begin{eqnarray*}
		\partial_\tau^1\left(\ee^{(t-\tau)\mathcal{L}}\alpha(\tau)\ee^{\tau\mathcal{L}}\right)&=& \partial_\tau^1\left(\ee^{(t-\tau)\mathcal{L}}\alpha(\tau)\right)\ee^{\tau\mathcal{L}}+\ee^{(t-\tau)\mathcal{L}}\alpha(\tau)\mathcal{L}\ee^{\tau\mathcal{L}}\\
		&=&(-1)\mathcal{L}\ee^{(t-\tau)\mathcal{L}}\alpha(\tau)\ee^{\tau\mathcal{L}}+\ee^{(t-\tau)\mathcal{L}}\alpha'(\tau)\ee^{\tau\mathcal{L}}+\ee^{(t-\tau)\mathcal{L}}\alpha(\tau)\mathcal{L}\ee^{\tau\mathcal{L}}\\
		&=& (-1)\ee^{(t-\tau)\mathcal{L}}ad^1_\mathcal{L}(\alpha(\tau))\ee^{\tau\mathcal{L}}+\ee^{(t-\tau)\mathcal{L}}\alpha'(\tau)\ee^{\tau\mathcal{L}}\\
		&=&\sum_{\ell=0}^{1}(-1)^\ell \binom{1}{\ell} \ee^{(t-\tau)\mathcal{L}}ad_\mathcal{L}^\ell\big(\alpha^{(1-\ell)}(\tau)\big)\ee^{\tau\mathcal{L}}.
	\end{eqnarray*}
	Suppose now the formula (\ref{equality_ady_time_dependent}) is valid for $k-1, \ k> 1$. Then, by induction, we have
	\footnotesize
	\begin{eqnarray*}
		\partial_\tau^k\left(\ee^{(t-\tau)\mathcal{L}}\alpha(\tau)\ee^{\tau\mathcal{L}}\right) &=&\partial_\tau^1\left(\partial_\tau^{k-1}\left(\ee^{(t-\tau)\mathcal{L}}\alpha(\tau)\ee^{\tau\mathcal{L}}\right)\right)\\
		&=&\sum_{\ell=0}^{k-1}(-1)^\ell \binom{k-1}{\ell} \partial_\tau^1\left(\ee^{(t-\tau)\mathcal{L}}ad_\mathcal{L}^\ell\big(\alpha^{(k-1-\ell)}(\tau)\big)\ee^{\tau\mathcal{L}}\right)\\
		&=&\sum_{\ell=0}^{k-1}(-1)^{\ell} \binom{k-1}{\ell}\left((-1)\ee^{(t-\tau)\mathcal{L}}ad^{\ell+1}_\mathcal{L}(\alpha^{(k-1-\ell)}(\tau))\ee^{\tau\mathcal{L}}+\ee^{(t-\tau)\mathcal{L}}ad_\mathcal{L}^\ell\big(\alpha^{(k-\ell)}(\tau)\big)\ee^{\tau\mathcal{L}}\right)\\
		&=&\sum_{\ell=0}^{k-2}(-1)^{\ell+1} \binom{k-1}{\ell}\ee^{(t-\tau)\mathcal{L}}ad^{\ell+1}_\mathcal{L}(\alpha^{(k-(\ell+1))}(\tau))\ee^{\tau\mathcal{L}}+(-1)^k\ee^{(t-\tau)\mathcal{L}}ad_\mathcal{L}^k(\alpha(\tau))\ee^{\tau\mathcal{L}}\\
		&+&\sum_{\ell=1}^{k-1}(-1)^{\ell} \binom{k-1}{\ell}\ee^{(t-\tau)\mathcal{L}}ad_\mathcal{L}^\ell\big(\alpha^{(k-\ell)}(\tau)\big)\ee^{\tau\mathcal{L}}+\ee^{(t-\tau)\mathcal{L}}\alpha^{(k)}(\tau)\ee^{\tau\mathcal{L}}\\
		&=&\sum_{\ell=1}^{k-1}(-1)^{\ell} \binom{k-1}{\ell-1}\ee^{(t-\tau)\mathcal{L}}ad^{\ell}_\mathcal{L}(\alpha^{(k-\ell)}(\tau))\ee^{\tau\mathcal{L}}+(-1)^k\ee^{(t-\tau)\mathcal{L}}ad_\mathcal{L}^k(\alpha(\tau))\ee^{\tau\mathcal{L}}\\
		&+&\sum_{\ell=1}^{k-1}(-1)^{\ell} \binom{k-1}{\ell}\ee^{(t-\tau)\mathcal{L}}ad_\mathcal{L}^\ell\big(\alpha^{(k-\ell)}(\tau)\big)\ee^{\tau\mathcal{L}}+\ee^{(t-\tau)\mathcal{L}}\alpha^{(k)}(\tau)\ee^{\tau\mathcal{L}}\\
		&=&\sum_{\ell=1}^{k-1}(-1)^{\ell} \binom{k}{\ell}\ee^{(t-\tau)\mathcal{L}}ad^{\ell}_\mathcal{L}(\alpha^{(k-\ell)}(\tau))\ee^{\tau\mathcal{L}}\\
		&+&\ee^{(t-\tau)\mathcal{L}}\alpha^{(k)}(\tau)\ee^{\tau\mathcal{L}}+(-1)^k\ee^{(t-\tau)\mathcal{L}}ad_\mathcal{L}^k(\alpha(\tau))\ee^{\tau\mathcal{L}}\\
		&=&\sum_{\ell=0}^{k}(-1)^{\ell} \binom{k}{\ell}\ee^{(t-\tau)\mathcal{L}}ad^{\ell}_\mathcal{L}(\alpha^{(k-\ell)}(\tau))\ee^{\tau\mathcal{L}}.
	\end{eqnarray*}
\normalsize
	In the penultimate equality we used the identity $\binom{k}{\ell} = \binom{k-1}{\ell}+\binom{k-1}{\ell-1}$.
\end{proof}

\end{document}